\documentclass[10pt,a4paper,reqno]{amsart}
\usepackage[T1]{fontenc}
\usepackage[utf8]{inputenc}
\usepackage{graphicx}
\usepackage{latexsym}
\usepackage{amssymb}
\usepackage{amsmath}
\usepackage{enumerate}

\usepackage{amsmath}
\usepackage{stmaryrd}
\usepackage{hyperref}
\usepackage{amssymb}
\usepackage{CJK}
\usepackage{color}
\usepackage{tikz}
\usepackage{amsmath,latexsym,amssymb,amsthm,array,amsfonts}
\usepackage{mathtools, stmaryrd}
\usepackage{xparse} \DeclarePairedDelimiterX{\Iintv}[1]{\llbracket}{\rrbracket}{\iintvargs{#1}}
\NewDocumentCommand{\iintvargs}{>{\SplitArgument{1}{,}}m}
{\iintvargsaux#1} %
\NewDocumentCommand{\iintvargsaux}{mm} {#1\mkern1.5mu..\mkern1.5mu#2}
\usepackage{mathrsfs,dsfont,bbm}
\usepackage{csquotes}
\usepackage{layout}
\usepackage{layout}
\newtheorem{theorem}{Theorem}[section]

\newtheorem{lemma}[theorem]{Lemma}
\newtheorem{definition}[theorem]{Definition}
\newtheorem{corollary}[theorem]{Corollary}
\newtheorem{proposition}[theorem]{Proposition}
\newtheorem{remark}[theorem]{Remark}

\def\real{{\mathord{{\rm I\kern-2.8pt R}}}}        
\def\inte{{\mathord{{\rm I\kern-2.8pt N}}}}

\def\sZZ{{\rm Z\kern-2.8ptem{}Z}}

\def\z{{\mathchoice
  {\sZZ}
  {\sZZ}
  {\rm Z\kern-0.30em{}Z}
  {\rm Z\kern-0.25em{}Z} }}
\def\sQQ{{\kern 0.27em \vrule height1.45ex width0.03em depth0em
          \kern-0.30em \rm Q}}
\def\qu{{\mathchoice
    {\sQQ}
    {\sQQ}
  {\kern 0.225em \vrule height1.05ex width0.025em depth0em \kern-0.25em \rm Q}
  {\kern 0.180em \vrule height0.78ex width0.020em depth0em \kern-0.20em \rm Q}
        }}
\def\sCC{{\kern 0.27em \vrule height1.45ex width0.03em depth0em
          \kern-0.30em \rm C}}
\def\complex{{\mathchoice
    {\sCC}
    {\sCC}
  {\kern 0.225em \vrule height1.05ex width0.025em depth0em \kern-0.25em \rm C}
  {\kern 0.180em \vrule height0.78ex width0.020em depth0em \kern-0.20em \rm C}
        }}

\MakeOuterQuote{"}

\newcommand{\ba}{\begin{array}}
\newcommand{\ea}{\end{array}}
\newcommand{\be}{\begin{equation}}
\newcommand{\ee}{\end{equation}}
\newcommand{\bea}{\begin{eqnarray}}
\newcommand{\eea}{\end{eqnarray}}
\newcommand{\beaa}{\begin{eqnarray*}}
\newcommand{\eeaa}{\end{eqnarray*}}

\def\z{\zeta}

\font\tenmath=msbm10 \font\sevenmath=msbm7 \font\fivemath=msbm5
\newfam\mathfam \textfont\mathfam=\tenmath
\scriptfont\mathfam=\sevenmath \scriptscriptfont\mathfam=\fivemath

\def \={{\buildrel {\rm (law)} \over =}}

\def\qed{ \hfill \vrule width.25cm height.25cm depth0cm\smallskip}

\newcommand{\basa}{\begin{assumption}}
\newcommand{\easa}{\end{assumption}}

\newcommand{\bas}{\begin{assum}}
\newcommand{\eas}{\end{assum}}

\newcommand{\R}{\mathbb{R}}

\newcommand{\C}{\mathbb{C}}

\usepackage{amsmath,amsthm}
\usepackage{amssymb}
\DeclareMathOperator{\Hess}{Hess}

\newcommand{\A}{\mathbb{C}\langle X_1,\dots,X_n\rangle}

\newcommand{\Dom}{\mathrm{Dom}}

\newcommand{\ignore}[1]{}
\begin{document}

\title{Obata's rigidity theorem in free probability}
\author{Charles-Philippe Diez$^{\dagger}$}
\thanks{
$^{\dagger}$ Technion – Israel Institute of Technology, Faculty of Mathematics, Technion City, Haifa
3200003, Israel. \textbf{chphdiez@campus.technion.ac.il}.}

\begin{abstract}
We establish a free analogue of Obata’s rigidity theorem. More precisely, Cheng and Zhou (2017) proved that on a weighted Riemannian manifold, the sharp spectral gap (Poincaré constant) is achieved only when the space splits isometrically off a one‑dimensional Gaussian factor, providing an infinite‑dimensional counterpart of Obata’s rigidity theorem. We obtain the corresponding phenomenon in free probability, extending it beyond the setting of analytic self‑adjoint potentials:

Assume a self-adjoint $n$--tuple $X=(X_1,\ldots,X_n)$ admits Lipschitz conjugate variables in the
sense of Dabrowski (2014).  Under a suitable non-commutative curvature--dimension condition, we
show that any non-zero saturator of Voiculescu’s free Poincar\'e inequality must be an affine
function of the generators.  Consequently, we deduce that the von Neumann algebra $M=W^\ast(X_1,\ldots,X_n)$
necessarily splits off a freely complemented semicircular component $W^*(Y_1)\simeq L^{\infty}([-2,2],\mu_{sc})$, which is also maximal amenable in $M$.

More generally, whenever the first eigenspace of the free Laplacian $\Delta:=\partial^*\bar\partial$ is finite-dimensional of
rank $r\ge 1$, our rigidity argument shows that these $r$ extremal directions form a free
semicircular family, yielding a free product decomposition with an $L(\mathbb{F}_r)$
factor. This provides a free analogue of this classical Gaussian splitting phenomenon and reveals a rigidity mechanism under non-commutative curvature.
\end{abstract}

\maketitle

\section{Introduction}

\subsection{Classical Rigidity and Stability in Spectral Geometry}\label{1.1}

A central theme in geometric analysis is to understand the structure of spaces
that \emph{maximize} or \emph{minimize} a
given geometric or analytic quantity within a natural class.  Extremal
phenomena often force strong rigidity, revealing for example product structures or
canonical geometric models.  A nice example, which serves as a guiding
analogy for our work, is the Cheeger--Gromoll splitting theorem
\cite{CheegerGromoll}: a complete Riemannian manifold with non-negative Ricci
curvature that contains a line (i.e.\ an infinite-length minimizing geodesic)
must split isometrically as a product with $\mathbb{R}$.  This result
is the starting point of a broader principle: extremal geometric behavior, whether in
diameter, curvature, isoperimetry, or spectral quantities, etc... typically forces
the underlying space to exhibit strong structural constraints, and comparison with a reference space.
\par\vspace{0.5em}
Rigidity in spectral geometry begins with Obata’s theorem \cite{Obata}, which
identifies the round sphere as the unique extremal case of the Lichnerowicz
estimate \cite{Lich}.  If a $n$-dimensional Riemannian manifold $(M^n,g)$ satisfies
$\mathrm{Ric}\ge (n-1)g$,
then the first non-zero eigenvalue $\lambda_1$ of the Laplace--Beltrami
operator obeys $\lambda_1 \ge n$, and equality forces $(M^n,g)$ to be isometric to the standard sphere
$(\mathbb{S}^n,g_{\mathrm{can}})$.  This principle, saying that sharp eigenvalue
bounds rigidly determine geometry, has driven a long line of developments:
Cheeger’s inequality \cite{Cheeger} links $\lambda_1$ to isoperimetry,
Payne--Weinberger \cite{PayneWeinberger} established sharp spectral gap
estimates for convex domains, and Cheng’s comparison theorems \cite{Cheng}
show how analytic inequalities encode global geometric information.

\medskip

The spectral gap is captured probabilistically by the Poincaré inequality: a probability measure $\mu$ satisfies it with constant $C_P>0$ if

\[
\int f^2\,d\mu \;\le\; C_P \int |\nabla f|^2\,d\mu
\quad\text{whenever }\int f\,d\mu=0,
\]

and the optimal constant is $C_P=1/\lambda_1$.  This inequality governs variance, mixing, and concentration, holds for many log-concave measures, and in one dimension is characterized by the Muckenhoupt criterion \cite{Mac}.  It sits within a hierarchy of functional inequalities: logarithmic Sobolev \cite{Gross}, hypercontractivity, transportation inequalities, with deep consequences in convex geometry \cite{BobkovLedoux,KLS}, ergodic theory, and high-dimensional probability \cite{Milman}.  The Bakry--Émery $\Gamma$-calculus \cite{be} unifies these ideas through the curvature--dimension condition $CD(K,N)$; in particular, $CD(1,\infty)$ implies an upper bound on the Poincaré constant $C_P\le 1$.  This synthetic curvature framework extends to non-smooth metric measure spaces through the work of Sturm \cite{SturmI,SturmII}, Lott--Villani \cite{LottVillani}, and Ambrosio--Gigli--Savaré \cite{AGS}.
\par\vspace{0.5em}
The rigidity phenomenon most relevant to our work is the theorem of Cheng and Zhou \cite{ChengZhou}, which shows that on spaces satisfying the curvature–dimension condition $CD(1,\infty)$, extremizers of the Poincaré inequality (equivalently, the sharp spectral gap $\lambda_1=1$) force a Gaussian splitting. Their result, extended to $\mathrm{RCD}(1,\infty)$ spaces by Gigli, Ketterer, Kuwada, and Ohta \cite{GKKO,KuwadaOhta}, identifies a one-dimensional Gaussian direction and yields an isometric and measure-theoretic product decomposition. For clarity, we recall the statement in the smooth setting:

\begin{theorem}[Cheng--Zhou \cite{ChengZhou}; Gigli--Ketterer--Kuwada--Ohta \cite{GKKO}]\label{cheng}
Let $(M,g,\mu)$ be a smooth weighted Riemannian manifold satisfying $CD(1,\infty)$.  
If there exists a centered non-zero $f\in W^{1,2}(M,\mu)$ such that

\[
\int_M f^2\,d\mu \;=\; \int_M |\nabla f|^2\,d\mu,
\]

then $(M,g,\mu)$ splits isometrically and measure-theoretically as

\[
(M,g,\mu)\ \simeq\ (\mathbb{R},|\cdot|,\gamma)\ \times\ (M',g',\mu'),
\]

where $\gamma$ is the standard Gaussian measure on $\mathbb{R}$.
\end{theorem}
\par\vspace{0.5em}
Recent work has also established the \emph{stability} of this rigidity phenomenon: measures whose spectral gap is nearly optimal must be quantitatively close to Gaussian.  This was shown by De Philippis--Figalli \cite{DeFiga}, and improved by Courtade--Fathi \cite{CourtF} using Stein’s method, and then extended to $\mathrm{RCD}(1,\infty)$ spaces by Bertrand--Fathi \cite{FatB} using the extension of Klartag’s needle decomposition \cite{Klart} to the RCD setting by Cavaletti--Mondino~\cite{Cav}.  Altogether, these results show that extremal and near-extremal functional inequalities encode important geometric information.  This classical picture, and especially the Cheng--Zhou splitting Theorem~\ref{cheng} therefore serves as the guiding analogy for our work, where we seek free analogues of such rigidity phenomena in the setting of free probability.

\subsection{Free Probability Framework}

Free probability, introduced by Voiculescu~\cite{Voic1,Voic2,Voic3,V,V6}, provides a non‑commutative analogue of classical probability in which independence is replaced by \emph{freeness}.  Random variables are modeled by operators in a tracial von Neumann algebra, and their joint distribution is encoded by mixed moments.  Voiculescu’s discovery of \emph{asymptotic freeness} \cite{Voiculescu1991,Voiculescu1992} established the important bridge between random matrix theory and operator algebras: independent Gaussian or unitary ensembles behave, in the large‑dimension limit, as freely independent non‑commutative random variables.  The semicircular law plays the role of the Gaussian distribution, and many classical analytic tools have free analogues, including the free Poincaré inequality \cite{Voipub,Dab10} (and the free logarithmic Sobolev inequality of Biane and Speicher~\cite{BS2,BLS}), expressed through the free difference quotient and closely tied to Voiculescu’s free entropy and free Fisher information \cite{Voic2,V,V6,VF}.
\par\vspace{0.5em}
Free entropy theory has had major consequences for the structure theory of free group factors $L(\mathbb{F}_n),\:n\geq2$.  One of the main achievement of Voiculescu was the proof the absence of Cartan subalgebras \cite{Voic2} in free groups factors: the impossibility of realizing free group factors via measured equivalence relations (via Feldman--Moore \cite{FeldmanMoore}), and Ge’s primeness theorem \cite{Ge}.  These results demonstrate the analytic power of free entropy; see Charlesworth--Nelson~\cite{CN} for a detailed survey.  Although we do not use these consequences directly, they form part of the analytic landscape motivating our work here.
\par\vspace{0.5em}
Several breakthrough developments in free stochastic analysis and non\mbox{-}commutative
transport have expanded the scope of the theory far beyond expectation.  Biane’s free
hypercontractivity~\cite{Biane}, the free Malliavin calculus of Biane--Speicher~\cite{BS},
and the free diffusion theory of Guionnet--Shlyakhtenko~\cite{GS,GS2} established deep
connections between free probability, stochastic analysis, and random matrix models.
Subsequent work by Dabrowski, Guionnet, Jekel, Shlyakhtenko, and others
\cite{DAB,Dab14,YGS,Jek1,Jek2,Jek3,Jek4,Jek5} produced major advances in free entropy
(including the equality of microstates and non\mbox{-}microstates free entropy), developed
free stochastic PDEs, and established foundational results in free transport, the free
Wasserstein manifold, and non\mbox{-}commutative Monge--Kantorovich duality.
\par\vspace{0.5em}
Free transport techniques have yielded powerful structural consequences for von Neumann
and $C^\ast$-algebras, including progress on the isomorphism problem for $q$-deformed free
group factors~\cite{GS}.  Jekel, Li, and Shlyakhtenko~\cite{Jek4} provided a variational
construction of free Gibbs laws and extended triangular transport to the $C^\ast$-algebraic
setting.  Jekel’s triangular transport~\cite{Jek1} led to a breakthrough on Popa’s
\emph{freely complementation (FC) problem}~\cite{Popa}, which asks whether every maximal
amenable subalgebra $B\subset L(\mathbb{F}_n)$ is freely complemented, i.e.\ whether
$L(\mathbb{F}_n)=B * N$ for some tracial von Neumann algebra $N\neq \mathbb{C}1$.
More recently, Boschert, Davis, and Hiatt~\cite{freely} proved that MASAs obtained by
``cutting and pasting’’ corners of generator MASAs are also freely complemented, which provide more general examples satisfying the \emph{(FC)} problem.

\subsection{Free Poincaré Inequalities}

Voiculescu first observed a free analogue of the classical Poincaré inequality in an unpublished
note~\cite{Voipub} (see also Dabrowski~\cite[Lemma~2]{DAB}).  This important discovery initiated a non‑commutative
spectral gap theory in which free difference quotients play the role of classical gradients and
provide an infinitesimal calculus adapted to non‑commutative random variables. A key observation is that the free Poincaré inequality
holds \emph{universally}, without any assumptions on the \emph{law} of such variables, in sharp contrast with the
classical situation.  In perfect analogy with the Gaussian case, the free Poincaré constant for a
semicircular system is $C_P=1$ as proved by Biane~\cite{BLS}, but beyond the semicircular law essentially nothing is
known about sharp constants for other non‑commutative distributions.
\par\vspace{0.5em}
In one dimension, Ledoux and Popescu~\cite{LP} developed a streamlined approach and studied the
behavior of the free Poincaré constant under smooth changes of variables.  Voiculescu’s inequality
asserts that for any compactly supported probability measure $\mu$,
\begin{equation}\label{FPIn}
    \mathrm{Var}_{\mu}(f)
    \;\le\;
    C\iint\!\left(\frac{f(x)-f(y)}{x-y}\right)^2\,d\mu(x)\,d\mu(y),
\end{equation}
with $C\le 2\rho(\mu)^2$.  This is the free analogue of the classical spectral gap estimate, with
the free difference quotient replacing the usual derivative.
\par\vspace{0.5em}
A recent contribution in the operator‑valued setting is due to Ito~\cite{Ito}, who characterized
$B$‑valued semicircular systems via a sharp $B$‑valued free Poincaré inequality, extending Biane’s
result \cite{BLS}.  Ito also showed that Voiculescu’s conjecture on the $B$‑valued free Poincaré
inequality from \cite{Aim} does not hold as originally stated.  The methods in~\cite{Ito} rely on
explicit Chebyshev‑polynomial computations and differ substantially from the functional‑analytic
approach that we will pursue here.
\par\vspace{0.5em}
Free Poincaré inequalities are also now known to have strong operator‑algebraic consequences. In particular, Dabrowski~\cite{DAB} proved that finite free Fisher information forces $W^{*}(X_{1},\ldots,X_{n})$ to be a factor without property~$\Gamma$ in the sense of Murray and von Neumann~\cite{vonN}. A key ingredient in the factoriality step is that the kernel of the free difference quotients consists only of scalars. For the non-$\Gamma$ part, Dabrowski establishes a refined mixed Poincar\'e–non-$\Gamma$ inequality: for 
$Z \in W^*(X_1,\dots,X_n) \cap \mathrm{Dom}(\bar\partial)$ (the domain of the free difference quotients), the variance of $Z$ is quantitatively controlled by its commutators with the generators $X_i$ and the conjugate variables, ensuring that all central sequences are trivial.

\subsection{Outline of our Results}

Free functional inequalities for free Gibbs measures with convex potentials are not well understood
outside the semicircular case (or in the case of separable potentials, which, however, reduce to
freeness).  Even the behavior of the free Poincaré constant under natural non‑commutative convexity
assumptions is largely unknown.  In contrast with the classical theory of uniformly log‑concave
measures $e^{-V}dx$ with $\nabla^2V\ge cI_n$, the free setting lacks a general curvature framework.
Our first main result fills part of this gap: under a suitable non‑commutative curvature condition,
the free Poincaré constant is controlled by the inverse of the convexity parameter.  Convexity is
here encoded through bounds on the non‑commutative Hessian, providing a robust analytic mechanism for
establishing functional inequalities in free probability.
\par\vspace{0.5em}
We extend these results to the framework of \emph{Lipschitz conjugate variables}, introduced in the important work of
Dabrowski~\cite{Dab14}.  This setting is remarkably flexible: it does not require the existence of a
potential, yet it retains enough analytic structure to support a weak-$L^{2}$ Bakry--Émery type argument (integration is always performed against the trace state~$\tau$, although we believe that a stronger version at the level of completely positive maps should also hold).
  The key commutation
relations between the generator and the free difference quotients, established in Dabrowski’s
seminal work~\cite{Dab14}, play a central role here, and our curvature condition can be formulated
entirely in terms of the conjugate variables.
\par\vspace{0.5em}
Our main result is an analytic–algebraic rigidity theorem.  Under a non‑commutative
curvature–dimension condition $CD(1,\infty)$, any extremizer of Voiculescu’s free Poincar\'e
inequality must be affine linear in the variables $(X_1,\ldots,X_n)$.  More precisely, if the
non‑commutative Hessian of the conjugate variables satisfies

\[
\mathscr{J}\xi\;\ge\;(1\otimes 1)\otimes I_n,
\]

then a free Poincaré inequality with constant less than or equal to $1$ holds, and any function saturating it must be
affine linear, exactly as in the classical (or free) case, where the first non‑zero eigenfunctions of the
Ornstein--Uhlenbeck operator are linear.  After normalization, this forces the appearance of a
semicircular direction.  By freeness, this yields a free product decomposition

\[
W^\ast(X_1,\ldots,X_n)\simeq W^\ast(Y_1)*N,
\]

where $W^\ast(Y_1)$ is the diffuse abelian algebra generated by the semicircular direction.  Since
$W^\ast(Y_1)$ is freely complemented, Popa’s theorem~\cite{Popa} implies that for $n\ge 2$ it is in
fact \emph{maximal amenable} inside $W^\ast(X_1,\ldots,X_n)$ (and a fortiori a MASA).  This provides a free analogue
of Obata’s theorem.
\par\vspace{0.5em}
Our computations strongly suggest that a full \emph{free Bakry--Émery theory} is within reach, with
potential connections and applications to random matrix models, von Neumann algebras, or non‑commutative geometry.
In Section~\ref{Open} we outline several directions for future research.  Classically, for a
sufficiently regular log‑concave measure $\mu\propto e^{-V}dx$, the (positive) generator $-\mathcal L_V$ has pure point spectrum with
finite multiplicities, and its eigenfunctions exhibit exponential decay~\cite{Davies,ReedSimonIV}.
The free Ornstein--Uhlenbeck operator enjoys analogous spectral properties (up to the absence of
exponential decay, which has no meaning here), and it is natural to ask whether such features
persist for more general free Gibbs states or tuples with Lipschitz conjugate variables.
\par\vspace{0.5em}
A deeper conceptual direction concerns the notion of \emph{free Ricci curvature}.
In the present work we remain in the ``flat'' setting of free difference quotients,
whose exact coassociativity implies the absence of any intrinsic geometric curvature.
All curvature therefore arise only from the potential, through the Jacobian
$\mathscr{J}\xi$ of the conjugate variables, in direct analogy with the classical
Bakry--Émery theory on $\mathbb{R}^n$. Beyond this flat framework, Dabrowski’s theory of \emph{almost co-associative}
derivations~\cite{Dab14v2} suggests a natural candidate for a genuinely geometric
\emph{free Riemmanian tensor} (and hence also a Ricci tensor), encoded by a defect tensor measuring the failure of
coassociativity.  Developing curvature--dimension estimates and rigidity phenomena
in this more general setting would amount to a non‑commutative Bakry--Émery theory,
of which free difference quotients form the coassociative, Ricci‑flat special case. We return to this perspective in
Section~\ref{freericci}, where we outline how almost
co-associativity provides a geometric curvature term in addition to the potential
curvature coming from the Jacobian of conjugates variables.

\section{Preliminaries}
\subsection{Tracial von Neumann algebras}
Throughout the paper we work in a tracial $W^\ast$--probability space
$(\mathcal{M},\tau)$, where $\mathcal{M}$ is a finite von Neumann algebra with separable predual, equipped with a faithful normal tracial state $\tau$.  We denote by $L^2(\mathcal{M},\tau)$ the Hilbert space completion of $\mathcal{M}$ with respect to the $L^2$--inner product defined by
\[
\langle x, y\rangle_{L^2(\mathcal{M})} := \tau(y^\ast x),
\]
linear in the first variable and conjugate--linear in the second, and we identify $\mathcal{M}$ with its image in $B(L^2(\mathcal{M},\tau))$ via left multiplication. We write $1 \in \mathcal{M} \subset L^2(\mathcal{M})$ for the canonical cyclic and separating vector. Since $\tau$ is tracial, the map $x\mapsto x^\ast$ extends uniquely to a conjugate--linear isometric involution on $L^2(\mathcal{M},\tau)$; we denote it by $J$.  Thus $J$ is the Tomita conjugation, the $L^2$--extension of the $^\ast$--operation, and satisfies $J = J^\ast = J^{-1}$.  

For $f\in L^2(\mathcal{M},\tau)$ we write $f^\ast := Jf$, and define
\[
\Re f := \tfrac12(f+f^\ast), \qquad
\Im f := \tfrac{1}{2i}(f-f^\ast),
\]
which are self--adjoint, i.e. $J\:\Re f=\Re f$ and $J\:\Im f=\Im f$, and satisfy
\[
\tau(\Re f)=\Re\,\tau(f), \qquad \tau(\Im f)=\Im\,\tau(f),
\]
where we still denote $\tau(x):=\langle x,1\rangle_{L^2(M)}$ as the canonical $L^2$-extension of the trace.

\medskip

Given self--adjoint elements $X_1,\ldots,X_n\in\mathcal{M}_{sa}$, we write $
M := W^\ast(X_1,\ldots,X_n)
$ for the von Neumann algebra they generate, and denote by $L^2(M)$ the
$L^2$--space obtained as the completion of $M$ with respect to the
$L^2$--norm induced by $\tau$.

\medskip

For a von Neumann subalgebra $M\subset\mathcal{M}$ we denote by
$M\bar\otimes M^{op}$ the von Neumann tensor product of $M$ with its opposite
algebra $M^{op}$ (i.e. the same $^\ast$--algebra but with reversed multiplication).  This
convention is natural for bimodules, since a right $M$--action corresponds to a
left action of $M^{op}$.  The algebra $M\bar\otimes M^{op}$ acts normally on
$L^2(M)\bar\otimes L^2(M^{op})$ by left multiplication on each leg.  Because
$\tau$ is tracial, we may identify $L^2(M^{op})$ with $L^2(M)$, and hence $
L^2(M)\,\bar\otimes\,L^2(M^{op})
\;\simeq\;
L^2(M)\,\bar\otimes_{1\otimes\mathcal{O}}\,L^2(M)$.

\medskip

More precisely, for each $n$ there are canonical identifications

\[\label{HS}
HS(L^2(M))^{\oplus_n}
\;\simeq\;
\big(L^2(M)\bar\otimes L^2(M)\big)^{\oplus_n}
\;\simeq\;
\big(L^2(M)\bar\otimes_{1\otimes\mathcal{O}} L^2(M^{op})\big)^{\oplus_n},
\]

where the second isomorphism uses traciality to place the right leg in the
opposite algebra.  Under the standard identification
$L^2(M)\otimes L^2(M)\simeq HS(L^2(M))$, the tensor $a\otimes b$ corresponds to
the finite-rank operator $x\mapsto a\,\tau(bx)$.  As real $M$--$M$ bimodules, $a(b\otimes c)d = ab\otimes cd$
and the real structure is given by

\[\label{hs}
(a\otimes b)^\dagger = b^\ast\otimes a^\ast.
\]
corresponding to adjointness of Hilbert-Schmidt operators.

\medskip

Finally, the identification $1\otimes\mathcal{O}$ between
$L^2(M)\otimes L^2(M^{op})$ and $L^2(M)\otimes L^2(M)$ comes from traciality.
Under this identification, $
(1\otimes\mathcal{O})(a\otimes b)=a\otimes b,
$with $b$ regarded as an element of $M^{op}$.  Since we will use this
identification throughout, we typically suppress the subscript and omit the
symbol $1\otimes\mathcal{O}$ when the context is clear.

\medskip

Finally, unless otherwise indicated, $\|\cdot\|$ denotes the operator norm on some
$B(\mathcal{H})$ (and on matrix amplifications when relevant).

\bigskip
\subsection{Free difference quotients and conjugate variables}
In this part, we mostly follow Dabrowski's exposure \cite{Dab10,Dab14}, but we also refer to the original paper of Voiculescu~\cite{V} and we finally recommend looking at Mai-Speicher\cite{MS} for more details on \emph{free} and \emph{cyclic} differential calculus.
\begin{flushleft}
Let $(X_1,\ldots,X_n)\in\mathcal{M}^n_{\mathrm{s.a.}}$ be a self-adjoint
$n$--tuple (always assumed \emph{algebraically free}), and set $M=W^\ast(X_1,\ldots,X_n)$ and
$M_0=\mathbb{C}\langle X_1,\ldots,X_n\rangle$.
\end{flushleft}
\begin{definition}[Partial free difference quotients]
The \emph{partial free difference quotients} are the derivations

\[
\partial_i : M_0 \longrightarrow L^2(M)\,\bar\otimes\,L^2(M^{op}),
\]

determined on generators by

\[
\partial_i(X_j)=\delta_{ij}.1\otimes 1.
\]

\end{definition}

\begin{definition}(Free Fisher information, Voiculescu~\cite[Definition~6.1]{V})
The \emph{free Fisher information} of $(X_1,\ldots,X_n)$ is

\[
\Phi^\ast(X_1,\ldots,X_n)
   := \sum_{i=1}^n \big\|\partial_i^\ast(1\otimes 1)\big\|_{L^2(M)}^2,
\]

whenever each $\partial_i^\ast(1\otimes 1)$ exists in $L^2(M)$; otherwise
we set $\Phi^\ast(X)=+\infty$.
\end{definition}

\begin{proposition}(Closability, Voiculescu~\cite[Lemma~3.3]{V})
If $\Phi^\ast(X_1,\ldots,X_n)<\infty$, then each $\partial_i$ is closable as
an unbounded operator

\[
L^2(M)\longrightarrow L^2(M)\,\bar\otimes\,L^2(M^{op}),
\]

and we denote the closures $\bar\partial_i$.  We write
$\bar\partial=(\bar\partial_1,\ldots,\bar\partial_n)$.
\end{proposition}

\begin{definition}(Conjugate variables, Voiculescu~\cite[Definition~3.1]{V})
If the adjoint value

\[
\xi_i := \partial_i^\ast(1\otimes 1)\in L^2(M)
\]

exists, it is called the \emph{conjugate variable} associated to $X_i$.
\end{definition}

\begin{remark}
The free difference quotients are \emph{real} derivations: for all $x\in M_0$,

\[
\partial_i(x^\ast) = (\partial_i(x))^{\dagger},
\]

where again $(a\otimes b)^{\dagger} := b^\ast \otimes a^\ast$.
Assuming each $\partial_i$ is closable as an operator
$L^2(M)\to L^2(M)\bar\otimes L^2(M^{op})$, let $\bar\partial_i$ denote its
closure with

\[
\Dom(\bar\partial) = \bigcap_{i=1}^n \Dom(\bar\partial_i).
\]
The left-hand inclusion being trivial, while the reverse inclusion follows from a duplicate adjoint argument, see Dabrowski~\cite{Dab10}.
\bigskip
Then $\Dom(\bar\partial)$ is $^\ast$‑stable, and

\[
\bar\partial_i(x^\ast) = \big(\bar\partial_i(x)\big)^{\dagger},
\qquad \forall\, x\in \Dom(\bar\partial).
\]

\end{remark}

More general notions of \emph{conjugate system}, involving weaker assumptions,
also exist as introduced below.

\begin{definition}(Conjugate system, Mingo--Speicher \cite[Definition~12]{SM})
Let $X=(X_1,\ldots,X_n)\in\mathcal{M}^n$.  
A family $\xi_1,\ldots,\xi_n\in L^2(\mathcal{M})$ is called a 
\emph{conjugate system} for $X$ if for every non-commutative polynomial 
$P\in M_0=\mathbb{C}\langle X_1,\ldots,X_n\rangle$ one has
\begin{equation}\label{conj-rel}
    \tau(\xi_iP(X))
    =
    \tau\otimes\tau(\partial_i P(X)),
    \qquad 1\le i\le n.
\end{equation}
\end{definition}

Strictly speaking, the requirement that the conjugate variables $\xi_i$ lie in 
$L^2(M)$, where $M=W^\ast(X_1,\ldots,X_n)$, is not essential.  
Indeed, even if one first constructs elements $\xi_i\in L^2(\mathcal{M})$ satisfying
\eqref{conj-rel}, one may obtain a genuine conjugate system inside $L^2(M)$ by
projecting onto $L^2(M)$:

\[
p:L^2(\mathcal{M})\longrightarrow L^2(M),
\qquad 
\xi_i\longmapsto p\xi_i.
\]

Since $p$ is the orthogonal projection onto $L^2(M)$, the relations
\eqref{conj-rel} remain valid for $p\xi_i$.  
This motivates the general definition above: the variables $\xi_i$ are required
only to satisfy the conjugate relations, regardless of whether they initially
belong to $L^2(M)$.

\begin{remark}
The conjugate relations determine the inner products of the $\xi_i$ with a dense
subset of $L^2(M)$; hence a conjugate system, if it exists, is unique.
\end{remark}

\medskip

Assuming closability of the partial free difference quotients (equivalently,
finite free Fisher information), one also obtains a non‑commutative Dirichlet
form with several additional properties, summarized below.

\begin{proposition}[Cipriani-Sauvageaot \cite{Sauvage,CiprianiSauvageot2003}, Peterson\cite{Peterson}, Dabrowski \cite{Dab10}]\label{resolvent}
Let $M=W^\ast(X_1,\ldots,X_n)$ be the tracial von Neumann algebra generated by
self-adjoint variables $(X_1,\ldots,X_n)$ in a tracial $W^\ast$‑probability
space $(\mathcal{M},\tau)$, and assume $\Phi^\ast(X)<\infty$. Then:

\begin{enumerate}
  \item The operator

\[
     \Delta := \partial^\ast\bar{\partial},
     \qquad \bar{\partial}=(\bar{\partial}_1,\ldots,\bar{\partial}_n),
  \]

  is the generator of a completely Dirichlet form on $L^2(M)$ (as proved in~\cite{Sauvage}):
  \newline
the amplification $\Delta\otimes I_n$ generates a Dirichlet form on
  $M_n(M)$.

  \item The multiple of the resolvent maps

\[
     \eta_\alpha := \alpha(\alpha+\Delta)^{-1},\qquad \alpha>0,
  \]

  are unital, tracial ($\tau\circ\eta_\alpha=\tau$), positive, completely
  positive, and contractive on both $L^2(M)$ and $M$ (see e.g.~\cite{Ma}).  Moreover (see e.g.~\cite[Proposition~2.5]{CiprianiSauvageot2003}),

\[
     \|x-\eta_\alpha(x)\|\le 2\|x\|,
     \qquad
     \|x-\eta_\alpha(x)\|_{L^2(M)}\longrightarrow 0
     \quad (\alpha\to\infty).
  \]

  \item The semigroup $\varphi_t=e^{-t\Delta}$ is strongly continuous with
  generator $-\Delta$, and the resolvent admits the integral representation (both understood as pointwise Riemann integral)

\[
     \eta_\alpha
     = \alpha\int_0^\infty e^{-\alpha t}\,\varphi_t\,dt.
  \]

  \item The square‑root resolvent (see~\cite{Peterson})

\[
     \zeta_\alpha := \eta_\alpha^{1/2}
     = \pi^{-1}\!\int_0^\infty \frac{t^{-1/2}}{1+t}\,
       \eta_{\alpha(1+t)/t}\,dt
  \]

  satisfies

\[
     \mathrm{Ran}(\eta_\alpha)=\Dom(\Delta)\subset \Dom(\bar{\partial}),
     \qquad
     \mathrm{Ran}(\zeta_\alpha)=\Dom(\Delta^{1/2})=\Dom(\bar{\partial}),
  \]

  so that $\bar{\partial}\circ\zeta_\alpha$ is bounded.
\end{enumerate}
\end{proposition}

For more precise background on non‑commutative Dirichlet forms, see the foundational work of
Cipriani–Sauvageot~\cite{Sauvage,CiprianiSauvageot2003}, as well as
Cipriani~\cite{Cipriani}. Applications to in von Neumann algebras theory was initiated Peterson~\cite{Peterson} (who introduced the notion of $L^2$-rigidity) and in several works of Dabrowski~\cite{Dab10,DAB,Dab14} 

\medskip

We record two basic properties of conjugate variables.

\begin{proposition}\label{self}
The conjugate variables $\xi_i\in L^2(M)$, $i=1,\ldots,n$, are self‑adjoint:

\[
J(\xi_i)=\xi_i,
\]

where $J$ denotes the Tomita conjugation operator.
\end{proposition}

\begin{proposition}[Voiculescu \cite{V6}]\label{ab}
Let $(X_1,\ldots,X_n)\in\mathcal{M}^n_{\mathrm{s.a.}}$ and let
$\xi_1,\ldots,\xi_n\in L^2(M)$ be a conjugate system for $X$.  
Then

\[
\sum_{i=1}^n [\xi_i,X_i]=0,
\]

where $[a,b]=ab-ba$ denotes the commutator.
\end{proposition}

This identity will play a key role when upgrading regularity properties of
conjugate variables under stronger assumptions.

\subsection{Free Poincaré Inequalities}

Voiculescu introduced in \cite{Voipub} a robust notion of \emph{free Poincaré inequalities}, later strengthened and developed by Dabrowski \cite{Dab10}. These inequalities have become a cornerstone of free probability, with far-reaching consequences for the structure of von Neumann algebras.

\begin{lemma}[Free Poincaré inequality, Voiculescu \cite{Voipub}]
Let $\partial_i$ denote the partial free difference quotient with respect to $X_1,\ldots,X_n$, and let $Y$ be a self-adjoint variable in the domain of all the operators $\bar{\partial}_i$ (viewed as unbounded operators)

Then there exists a positive constant $C$, depending on the $X_i$ but not on $Y$, such that

\[
\big\|\, Y - \tau(Y).1\,\big\|_{L^2(M)}^2 \;\leq\; C \sum_{j=1}^n \big\| \bar{\partial}_j Y \big\|_{L^2(M){\bar\otimes} L^2(M^{op})}^2 .
\]

Moreover, the best (smallest) constant $C_P$, called the \textbf{free Poincaré constant}, satisfies the bound

\[
C_P \;\leq\; 4n\max_{i=1,\ldots,n}\lVert X_i\rVert^2.
\]

\end{lemma}

This result is striking: it holds for \emph{any} choice of $(X_1,\ldots,X_n)$, in sharp contrast with the classical situation as evoked above.

\subsection{The Gibbs Case and Lipschitz Conjugate Variables}

The free Gibbs state model provides a natural testing ground for these ideas. In this setting, the conjugate variables are especially well behaved: they are bounded operators, in fact explicit polynomials in the Gibbs tuple. Concretely, given a self-adjoint non-commutative polynomial potential $V$, the free Gibbs law $\tau_V$ is defined as follows (note that we don't define properly what is the non-microstates free entropy and we refer to the aforementioned papers for more details, especially~\cite{Voic2}):

\begin{definition}[Free Gibbs law, \cite{GM,FN,VF}]
The free Gibbs law $\tau_V$ associated with the potential $V$, if it exists, is the minimizer of the functional
\begin{equation}
\chi_V(\tau'):=-\chi(\tau')+\tau'(V),
\end{equation}
where $\chi(\tau')$ denotes Voiculescu’s microstates free entropy, among tracial states $\tau'$ on the free $*$-algebra $\mathscr{P}:=\mathbb C\langle t_1,\ldots,t_n\rangle$: the $*$-algebra of non-commutative polynomials in $n$-self-adjoint non-commuting (formal) variables $t_1,\ldots,t_n$.
\end{definition}

A variational argument of Voiculescu~\cite[Section~3.7]{VF} (see also Fathi-Nelson~\cite[Lemma~1.6]{FN}), using his infinitesimal \emph{change-of-variables} formula for microstates free entropy, shows that such a critical state $\tau_V$ when $V=V^*\in \mathscr P$ necessarily satisfies the \emph{Schwinger--Dyson equation}:
\begin{equation}
    \tau_V(\mathscr{D}_i VP)\;=\;\tau_V\otimes \tau_V(\partial_i P),
\end{equation}
for all non-commutative polynomials $P\in \mathscr{P}$. 
where for each $i=1,\ldots,n$, and each monomial $q\in \mathscr{P}$, 
\begin{equation}
    \mathscr{D}_iP=\sum_{P=AX_iB}BA
\end{equation}
and then extended by linearity. We then denote the \emph{cyclic gradient} $\mathscr{D}q=(\mathscr{D}_1q,\ldots,\mathscr{D}_nq)$.
\begin{flushleft}
Thus, the conjugate variables (assuming the potential $V=V^*$ is self-adjoint) are given explicitly by

\[
\xi_i := \mathscr{D}_i V(X_1,\ldots,X_n)\in \mathbb{C}\langle X_1,\ldots,X_n\rangle \subset W^*(X_1,\ldots,X_n)\subset L^2(W^*(X_1,\ldots,X_n)),
\]

and

\[
\partial_j \xi_i = \partial_j \mathscr{D}_i V \in \mathbb{C}\langle X_1,\ldots,X_n\rangle \otimes \mathbb{C}\langle X_1,\ldots,X_n\rangle^{op} \subset W^*(X_1,\ldots,X_n)\,\bar{\otimes}\,W^*(X_1,\ldots,X_n)^{op}.
\]

The existence of such Gibbs states is a subtle problem, known until recently only for small semicircular perturbations of quadratic potentials (Guionnet--Maurel-Ségala \cite{GM}). A breakthrough was however recently achieved independently by Dabrowski, Guionnet, and Shlyakhtenko \cite{YGS}, and Jekel \cite{Jek3} who introduced two suitable (but different) notion of convexity ensuring existence and uniqueness of such free Gibbs laws. Jekel, Li and Shlyakhtenko also proved existence of free Gibbs laws by a purely variational argument in~\cite[Section~7]{Jek4} (while still using the SDE approach to obtain uniqueness for the convex setting).
\end{flushleft}

\bigbreak
While having cyclic gradients of nice analytic potentials as conjugate variables is an extremely nice situation, it remains a rather restrictive assumption, which motivates the search for a more flexible framework. More generally, for an arbitrary $n$\,--tuple, the requirement that $
\xi_i \in \Dom(\bar{\partial}) := \bigcap_{j=1}^n \Dom(\bar{\partial}_j)
$, together with the regularity condition
$\bar{\partial}_j \xi_i \in M \,\bar{\otimes}\, M
$
is an intermediate regime, which provides the right framework for our purpose. It is reminiscent of the free Gibbs situation and, in particular, of the semicircular case on the free group factors $L(\mathbb{F}_n)$, $n \ge 1$. This condition was formally introduced by Dabrowski~\cite{Dab14} under the name \emph{Lipschitz conjugate variables}, in reference to the Sobolev-type characterization of Lipschitz functions in the one-variable setting. The terminology reflects the underlying intuition that the free derivatives of such an element (which may be a priori unbounded; however, the next Proposition~\ref{prop5} rules out this possibility) behave as bounded operators.

    \begin{definition}(Dabrowski \cite[Definition 1]{Dab14})\label{def6}
Let \({M} = (W^*(X_1,\ldots,X_n), \tau)\) be a tracial von Neumann algebra.  
We say that \(M\) satisfies the \emph{Lipschitz conjugate variable condition} if:

\begin{enumerate}
  \item the partial free difference quotients \(\partial_i\) are defined and closable,
  \item the conjugate variables $\xi_i:=\partial_i^*(1 \otimes 1)$ exist in \(L^2(M)\) for all \(i=1,\ldots,n\),
and moreover, these conjugate variables are in the domain of the closure 
  \(\bar{\partial} = (\bar{\partial}_1,\ldots,\bar{\partial}_n)\), with \[
     \bar{\partial}_j \, \partial_i^*(1 \otimes 1) \;\in\; {M} \bar{\otimes} {M}^{op} 
     \;\subset\; L^2({M} \bar{\otimes} {M}^{op}) \;\cong\; L^2({M}) \bar{\otimes} L^2({M}^{op}) \;\cong\; L^2({M}) \bar{\otimes} L^2({M}),
  \] for each $j=1,\ldots,n$.
\end{enumerate}
\end{definition}

\begin{flushleft}
   Let us now notice a remarkable and somewhat unexpected fact about Lipschitz conjugate variables, first observed by Dabrowski in \cite[Sketch of Proof of
Corollary~25~(a')]{Dab14}.  
The first requirement in the above definition, namely the condition
$\xi_i \in L^2(M)$, can in fact be substantially strengthened to $\xi_i \in M$. 
This could already be deduced from the \emph{commutator} identity for the free
difference quotient established by Voiculescu (see Dabrowski~\cite[Equality~(1)]{Dab10}), together
with a density argument.  
However, we present here an alternative proof based instead on a careful
manipulation of Voiculescu’s fundamental \emph{commutation} relation between conjugates variables and generators
\eqref{ab}, which provides a more direct conceptual explanation of the phenomenon.

\end{flushleft}
\begin{proposition}\label{prop5}
Suppose $M=(W^\ast(X_1,\ldots,X_n),\tau)$ satisfies the \emph{Lipschitz conjugate variable condition}. Then:
\begin{enumerate}
  \item $\xi_i:=\partial_i^\ast(1\otimes 1)\in M$ for all $i$.
  \item $\xi_i^{(2)}:=\partial_i^\ast(1\otimes \partial_i^\ast(1\otimes 1))=\partial_i^\ast( 1\otimes \xi_i)=\partial_i^\ast(\xi_i\otimes 1)\in M$ for all $i$.
  \item $(\partial_i\otimes \mathrm{id})\circ \partial_j:\ M_0=\A\to L^2(M)^{\bar\otimes 3}$ is densely defined and closable for all $i,j=1,\ldots,n$.
\end{enumerate}
\end{proposition}

\begin{proof}
\begin{enumerate}
    \item For the first point.
Recall that Voiculescu’s $L^2$ commutation Proposition~\ref{ab} gives

\[
\sum_{i=1}^n [\xi_i,X_i]\ =\ 0 \qquad\text{in }L^2(M).
\]

Fix $j\in\{1,\dots,n\}$.
Since $\xi_i\in \Dom(\bar\partial_j)$, there exist polynomials $p_k^{(i)}\in\A$ such that

\[
p_k^{(i)}\to \xi_i \ \text{ in }L^2(M),
\qquad
\partial_j p_k^{(i)}\to \bar\partial_j\xi_i \ \text{ in }L^2(M)\,\bar\otimes\,L^2(M^{op}),
\]

for every $i$.
Define, for each $k$,

\[
a_k := \sum_{i=1}^n [p_k^{(i)},X_i]
     = \sum_{i=1}^n\big(p_k^{(i)}X_i - X_i p_k^{(i)}\big)\ \in \A.
\]

Since each $X_i\in M$ is bounded, left and right multiplication by $X_i$ define bounded
operators on $L^2(M)$ with norm at most $\|X_i\|$. Thus, for every $i$,

\[
\|p_k^{(i)}X_i - \xi_i X_i\|_{L^2(M)}
 \le \|X_i\|\,\|p_k^{(i)}-\xi_i\|_{L^2(M)}\ \longrightarrow\ 0,
\]

and similarly $\|X_i p_k^{(i)} - X_i\xi_i\|_{L^2(M)}\to 0$.
Hence

\[
[p_k^{(i)},X_i] \longrightarrow [\xi_i,X_i]\quad\text{in }L^2(M),
\]

and summing over $i$ (finite sum) gives

\[
a_k \longrightarrow \sum_{i=1}^n [\xi_i,X_i] = 0
\qquad\text{in }L^2(M).
\]

On the core $\A$, the Leibniz rule for $\partial_j:\A\to L^2(M)\,\bar\otimes\,L^2(M^{op})$ yields

\[
\partial_j(p_k^{(i)}X_i)
  = \partial_j(p_k^{(i)})(1\otimes X_i) + (p_k^{(i)}\otimes 1)\,\partial_j X_i,
\]

\[
\partial_j(X_i p_k^{(i)})
  = \partial_j X_i\,(1\otimes p_k^{(i)}) + (X_i\otimes 1)\,\partial_j p_k^{(i)}.
\]

Therefore,

\[
\partial_j\big([p_k^{(i)},X_i]\big)
 = \partial_j(p_k^{(i)}X_i - X_i p_k^{(i)})
\]

\[
 = \partial_j(p_k^{(i)})(1\otimes X_i) - (X_i\otimes 1)\,\partial_j p_k^{(i)}
   \;+\; (p_k^{(i)}\otimes 1)\,\partial_j X_i - \partial_j X_i\,(1\otimes p_k^{(i)}).
\]

Since $\partial_j X_i = \delta_{ij}(1\otimes 1)$ in $L^2(M)\,\bar\otimes\,L^2(M^{op})$, this simplifies to

\[
\partial_j\big([p_k^{(i)},X_i]\big)
 = \partial_j(p_k^{(i)})(1\otimes X_i) - (X_i\otimes 1)\,\partial_j p_k^{(i)}
   \;+\; \delta_{ij}\big(p_k^{(j)}\otimes 1 - 1\otimes p_k^{(j)}\big),
\]

where we use the convention

\[
[p_k^{(j)},1\otimes 1] := p_k^{(j)}\otimes 1 - 1\otimes p_k^{(j)}\in \mathbb{C}\langle X_1,\ldots,X_n\rangle \otimes \mathbb{C}\langle X_1,\ldots,X_n\rangle^{op}.
\]

Summing over $i$ we obtain
\begin{align*}
\partial_j(a_k)
 &= \sum_{i=1}^n\partial_j\big([p_k^{(i)},X_i]\big) \\
 &= \sum_{i=1}^n\Big(\partial_j p_k^{(i)}(1\otimes X_i)
                     - (X_i\otimes 1)\,\partial_j p_k^{(i)}\Big)
    +  \big(p_k^{(j)}\otimes 1 - 1\otimes p_k^{(j)}\big)\\
    &=\sum_{i=1}^n\Big(\partial_j p_k^{(i)}(1\otimes X_i)
                     - (X_i\otimes 1)\,\partial_j p_k^{(i)}\Big)+ [p_k^{(j)},1\otimes 1]
\end{align*}

We now show that $\partial_j(a_k)$ converges in $L^2(M)\,\bar\otimes\,L^2(M^{op})$ and identify its limit.
First, for the “gradient” terms, note that for each $i$ the maps

\[
L^2(M)\,\bar\otimes\,L^2(M^{op})\ni \eta \mapsto (X_i\otimes 1)\eta,\qquad
\eta\mapsto \eta(1\otimes X_i)
\]

are bounded linear operators. Indeed, under the usual identification
$L^2(M)\,\bar\otimes\,L^2(M^{op}) \simeq L^2(M\bar\otimes M^{op})$, they are just left
and right multiplication by $X_i\otimes 1$ and $1\otimes X_i$, respectively, so their
operator norms are bounded by $\|X_i\otimes 1\|=\|X_i\|$ since $X_i\in M$ is bounded.
\newline
Since $\partial_j p_k^{(i)}\to \bar\partial_j\xi_i$
in $L^2(M)\,\bar\otimes\,L^2(M^{op})$, we have

\[
\partial_j p_k^{(i)}(1\otimes X_i)
 \longrightarrow \bar\partial_j\xi_i(1\otimes X_i),\qquad
(X_i\otimes 1)\,\partial_j p_k^{(i)}
 \longrightarrow (X_i\otimes 1)\,\bar\partial_j\xi_i
\]

in $L^2(M)\,\bar\otimes\,L^2(M^{op})$, so

\[
\partial_j p_k^{(i)}(1\otimes X_i) - (X_i\otimes 1)\,\partial_j p_k^{(i)}
 \longrightarrow \bar\partial_j\xi_i(1\otimes X_i) - (X_i\otimes 1)\,\bar\partial_j\xi_i
\]

in $L^2(M)\,\bar\otimes\,L^2(M^{op})$.
\newline
Summing over $i$ again (finite sum) yields

\[
\sum_{i=1}^n\Big(\partial_j p_k^{(i)}(1\otimes X_i)
                 - (X_i\otimes 1)\,\partial_j p_k^{(i)}\Big)
 \longrightarrow
\sum_{i=1}^n\Big(\bar\partial_j\xi_i(1\otimes X_i)
                 - (X_i\otimes 1)\,\bar\partial_j\xi_i\Big)
\]

in $L^2(M)\,\bar\otimes\,L^2(M^{op})$.
\newline
Next, for the other term, consider the linear map

\[
T:L^2(M)\to L^2(M)\,\bar\otimes\,L^2(M^{op}),\qquad T(x):=x\otimes 1 - 1\otimes x.
\]

A direct computation shows that $T$ is an almost (up to a $\sqrt{2}$-factor) isometry on the centered subspace $L^2_0(M)$, i.e.

\[
\|T(x)\|_{L^2(M)\,\bar\otimes\,L^2(M^{op})}^2
 = (\tau\otimes\tau)\big((x\otimes 1-1\otimes x)^\ast(x\otimes 1-1\otimes x)\big)
 = 2\|x-\tau(x).1\|_{L^2(M)}^2,
\]

so in particular $T$ is continuous and

\[
\|T(x)\|_{L^2(M)\,\bar\otimes\,L^2(M^{op})}
 \le \sqrt{2}\,\|x\|_{L^2(M)},\qquad x\in L^2(M).
\]

Since $p_k^{(j)}\to \xi_j$ in $L^2(M)$, we obtain

\[
p_k^{(j)}\otimes 1 - 1\otimes p_k^{(j)} = T(p_k^{(j)}) \longrightarrow T(\xi_j)
 = \xi_j\otimes 1 - 1\otimes \xi_j
\]

in $L^2(M)\,\bar\otimes\,L^2(M^{op})$.

Combining these convergences (all being in $L^2(M)\,\bar\otimes\,L^2(M^{op})$), we conclude that

\[
\partial_j(a_k)
 \longrightarrow 
\eta := \sum_{i=1}^n\Big(\bar\partial_j\xi_i(1\otimes X_i)
                         - (X_i\otimes 1)\,\bar\partial_j\xi_i\Big)
       \;+\; \big(\xi_j\otimes 1 - 1\otimes \xi_j\big)
\]

in $L^2(M)\,\bar\otimes\,L^2(M^{op})$ (in fact it belongs to $M\bar\otimes M^{op}$).

We have thus shown:

\[
a_k\to \sum_{i=1}^n[\xi_i,X_i] \ \text{ in }L^2(M),
\qquad
\partial_j(a_k)\to \eta \ \text{ in }L^2(M)\,\bar\otimes\,L^2(M^{op}).
\]

By definition of the closure $\bar\partial_j$ of $\partial_j$, this implies that
$\sum_{i=1}^n [\xi_i,X_i]\in \Dom(\bar\partial_j)$ and

\[
\bar\partial_j\Big(\sum_{i=1}^n [\xi_i,X_i]\Big)=\eta.
\]

Since $\sum_{i=1}^n[\xi_i,X_i]=0$ in $L^2(M)$, the left-hand side is $\bar\partial_j(0)=0$,
and hence $\eta=0$. Therefore,

\[
\xi_j\otimes 1 - 1\otimes \xi_j
 = \sum_{i=1}^n\Big((X_i\otimes 1)\,\bar\partial_j\xi_i
                    - \bar\partial_j\xi_i(1\otimes X_i)\Big),
\]

In particular, by the Lipschitz hypothesis, each $\bar\partial_j\xi_i\in M\bar\otimes M^{op}$,
so the right-hand side lies in $M\bar\otimes M^{op}$, whence also
$\xi_j\otimes 1 - 1\otimes \xi_j\in M\bar\otimes M^{op}$.
Applying $\mathrm{id}\otimes\tau$ yields a kind of non-commutative Clark–Ocone–Bismut formula (in analogy with the one in free Malliavin calculus on Wigner space as developped by Biane-Speicher\cite{BS})

\[
\xi_j - \tau(\xi_j).1
 = (\mathrm{id}\otimes\tau)\!\left[
     \sum_{i=1}^n\Big((X_i\otimes 1)\,\bar\partial_j\xi_i
                     - \bar\partial_j\xi_i(1\otimes X_i)\Big)
   \right]\in M.
\]

Thus $\xi_j\in M$ for all $j$.
\item For the second point. Voiculescu’s adjoint identity \cite[Proposition~4.1]{V} gives

\[
\xi_i^{(2)} := \partial_i^\ast(1\otimes \xi_i)
 = \xi_i^2 - (\tau\otimes\mathrm{id})\big(\bar\partial_i(\xi_i)\big)\in M,
\]

since $\xi_i\in M$ and $\bar\partial_i(\xi_i)\in M\bar\otimes M^{op}$ by the Lipschitz assumption.
\newline
Moreover, it is well known (see e.g.\ Mai \cite[Lemma~4.6]{Mai}) that if 
$u\in \Dom(\partial_i^\ast)$, then $u^{\dagger}\in \Dom(\partial_i^\ast)$ and, moreover,
\begin{equation}
    \partial_i^\ast(u^{\dagger}) \;=\; \partial_i^\ast(u)^\ast,
    \qquad \text{for every } i \in \{1,\ldots,n\}.
\end{equation}
In particular, this shows (by another way) that the conjugate variables are self-adjoint. Indeed,
\begin{equation}
    \xi_i^\ast 
    = \big(\partial_i^\ast(1\otimes 1)\big)^\ast
    = \partial_i^\ast\big((1\otimes 1)^{\dagger}\big)
    =\ \partial_i^\ast(1\otimes 1)
    = \xi_i,
    \qquad \text{for every } i \in \{1,\ldots,n\}.
\end{equation}

Furthermore, since $1\otimes \xi_i \in \Dom(\partial_i^\ast)$, it follows that

\[
(1\otimes \xi_i)^{\dagger} = \xi_i^\ast \otimes 1 = \xi_i \otimes 1 
\in \Dom(\partial_i^\ast),
\]

and hence
\begin{equation}
    \partial_i^\ast(\xi_i\otimes 1) = \partial_i^\ast(1\otimes \xi_i).
\end{equation}

\medskip

Since $\xi_i$ is self-adjoint, it only remains to check the second term in Voiculescu’s adjoint identity,

\[
\xi_i^{(2)} = \partial_i^\ast(1\otimes \xi_i)
= \xi_i^2 - (\tau\otimes \mathrm{id})\big(\bar\partial_i(\xi_i)\big).
\]

By Mai’s \cite[Lemma~4.6]{Mai} again, for any $u\in \Dom(\partial^\ast)$ one has

\[
(\tau\otimes \mathrm{id})(u)^\ast = (\mathrm{id}\otimes\tau)(u^\dagger).
\]

Applying this to $u=\bar\partial_i(\xi_i)$, and using that $\bar\partial_i$ is a real derivation so that 
$(\bar\partial_i \xi_i)^\dagger = \bar\partial_i(\xi_i^\ast)=\bar\partial_i(\xi_i)$, we deduce

\[
\big((\tau\otimes \mathrm{id})(\bar\partial_i \xi_i)\big)^\ast
= (\mathrm{id}\otimes\tau)(\bar\partial_i(\xi_i))=(\tau\otimes \mathrm{id})(\bar\partial_i\xi_i).
\]

Thus both $(\tau\otimes \mathrm{id})(\bar\partial_i(\xi_i))$ and 
$(\mathrm{id}\otimes\tau)(\bar\partial_i(\xi_i))$ are self-adjoint elements of $M$.

\medskip
  
Since $\xi_i$ is self-adjoint and the correction term 
$(\tau\otimes \mathrm{id})(\bar\partial_i(\xi_i))$ is also self-adjoint, 
it follows that the second-order conjugate variable

\[
\xi_i^{(2)} = \partial_i^\ast(1\otimes \xi_i)\in M
\]

is itself self-adjoint.

\item Finally,
since $\xi_i,\xi_i^{(2)}\in M$, the first and second-order conjugate variables of $X_i$ exist in $M$. By \cite[Lemma~40]{YGS}, the existence of first and second-order conjugate variables in $M$ implies that all second-order free difference quotients $(\partial_i\otimes \mathrm{id})\circ \partial_j$ are weak-$*$ closable on $M$ and $L^2$-closable on $L^2(M,\tau)$, both in the diagonal case $i=j$ and in the mixed case $i\neq j$.
\end{enumerate}
\end{proof}
\qed
\begin{remark}
One might wonder why we do not simply apply $\bar\partial_j$ to the identity
$\sum_i [\xi_i,X_i]=0$ and expand using the Leibniz rule.  The issue is that,
at this stage, we only know that $\xi_i\in L^2(M)\cap\Dom(\bar\partial_j)$, and
our goal is precisely to prove that $\xi_i\in M$.  Thus we cannot assume that
the commutators $[X_i,\xi_i]$ lie in $\Dom(\bar\partial_j)\cap M$ which is a $*$- algebra, on which $\bar\partial_j$ act as a derivation, see e.g. Dabrowski~\cite[Proposition~6]{Dab10}. Since we do not yet know that $\xi_i$ is bounded, this criterion does not apply directly.
\end{remark}

\begin{flushleft}
In Riemannian geometry, curvature manifests itself through commutation identities between the
Laplacian and the gradient.  A fundamental example is the following relation (stated for simplify in $\mathbb R^n$),

\[
[\nabla, \mathcal{L}_V] = -\,\nabla^2 V,
\]

where $\mathcal{L}_V$ is the Langevin generator associated with a potential $V$ (typically convex, so
$\nabla^2 V\ge 0$). Note also that it is readily extended to   This identity, discussed for instance in Bakry--Gentil--Ledoux
\cite[p.~104]{Ledoux}, underlies the decomposition of $\Gamma$--operators and plays a central role
in the Bakry--Émery curvature--dimension theory.  In particular, the Hessian $\nabla^2 V$ acts as a
curvature term, directly analogous to the Ricci tensor in the unweighted case.  The explicit form of
$\mathcal{L}_V$ is recalled in~\eqref{langevin}, where it appears as the generator of the Langevin diffusion
with invariant log-concave measure $\mu \propto e^{-V}dx$.

\medskip

In the free setting, an analogous phenomenon occurs: curvature is encoded in the failure of exact
commutation between the free difference quotients and the associated free Laplacian.  We now recall
a key lemma establishing this almost-commutation relation.  It should be viewed as the free
counterpart of the classical identity above, and it plays an equally fundamental role in the
structure of free $\Gamma$--calculus and in the analysis of free diffusion generators.

\end{flushleft}

\begin{lemma}\label{in2deltaDelta}(Dabrowski \cite[Lemma~19]{Dab14}). Under the assumptions in Definition~\ref{def6}, we have,
\begin{enumerate}
\item
Set $\Delta_j:=\partial_j^*\bar\partial_j$. For any $x\in \Dom(\partial)$ we have 
$x\in \Dom(\Delta_j)$, 
$\partial_i(x)\in \Dom(\Delta_j\otimes1+1\otimes \Delta_j)$, $x\in \Dom(\Delta^{3/2})$ and~:
 \begin{align*} 
&{\partial}_{i}\Delta_{j}(x)= (1 \otimes \Delta_j+\Delta_{j}\otimes 1){\partial}_i(x)
+\partial_{j}(x)\#(\bar{\partial}_{i}\partial_{j}^{*}(1\otimes 1) ).
\end{align*}
\item
 If $x\in \Dom(\bar{\partial})$ (resp. $x\in \Dom(\Delta)$)  then so is $\mathrm{id}\otimes\tau({\bar{\partial}_i}(x))$. 
\item
$\Dom(\Delta^{3/2})\subset \Dom(\overline{\Delta \otimes\mathrm{id}+\mathrm{id}\otimes \Delta }\circ\bar{\partial} )$ and moreover we have for any $x\in \Dom(\Delta^{3/2})$ \begin{align*} 
&\bar{\partial}_{i}\Delta(x)= \Delta^{\otimes}\bar{\partial}_i(x)
+\sum_{j=1}^N\bar{\partial}_{j}(x)\#(\bar{\partial}_{i}\partial_{j}^{*}(1\otimes 1) ).
\end{align*}
\end{enumerate}
We also denote in an shorthand $\Delta^{\otimes}:=\overline{\Delta\otimes \mathrm{id}+\mathrm{id}\otimes \Delta}$ which
 is thus a densely defined closed self-adjoint positive operator: it is the closure of $\Delta\otimes \mathrm{id}+\mathrm{id}\otimes \Delta$ defined on the core $\Dom(\Delta)\otimes \Dom(\Delta)$  (algebraic tensor product) and is also stable by $\varphi_t\otimes \varphi_t$.
\end{lemma}

\begin{flushleft}
Under this definition, we will show that several unexpected and remarkably robust properties continue to hold. In particular, we obtain a refined Poincar\'e inequality whose optimal constant is bounded by the operator norm of the non-commutative Jacobian matrix of the free difference quotient. This may be viewed as a free \emph{H\"ormander-Brascamp--Lieb} type inequality~\cite{Brascamp}, but now in a far more general setting, with no structural assumptions on the underlying state (in particular, no Gibbsian form), beyond the Lipschitz regularity of the conjugate variables.
\end{flushleft}
\begin{flushleft}
Before going further, we need to introduce the right--leg (``sharp'') action at the matrix level which will be crucial and unavoidable from the almost--commutation Lemma~\ref{in2deltaDelta}.
\end{flushleft}

\begin{definition}\label{def:right-action-RS-explicit}
Let $(M,\tau)$ be a tracial $W^\ast$–probability space. On the algebraic tensor product
$L^2(M)\otimes L^2(M^{op})$ we define, for $x\otimes y\in M\otimes M^{op}$ and
$a\otimes b\in L^2(M)\otimes L^2(M^{op})$,

\[
(a\otimes b)\sharp(x\otimes y) := a x\otimes y b.
\]

This extends by bilinearity and continuity to a bounded right action of the von Neumann tensor
product $M\bar\otimes M^{op}$ on $L^2(M){\bar\otimes}\,L^2(M^{op})$.

For $n\in\mathbb N$ and $T := (T)_{ij}=(T_{i,j})_{i,j=1}^n\in M_n(M\bar\otimes M^{op})$, we define the following linear operator called the
\emph{right-leg} multiplication (action)

\[
\mathcal R_T:\ \big(L^2(M){\bar\otimes}\,L^2(M^{op})\big)^{\oplus_n}\ \longrightarrow\ 
\big(L^2(M){\bar\otimes}\,L^2(M^{op})\big)^{\oplus_n}
\]

by
\begin{equation}
(\mathcal R_T\eta)_i := \sum_{j=1}^n \eta_j\sharp T_{ji}
\end{equation}
for $\eta = (\eta_1,\dots,\eta_n)\in \big(L^2(M){\bar\otimes}\,L^2(M^{op})\big)^{\oplus_n}$ and
$i=1,\dots,n$. 
\bigbreak
\end{definition}
\begin{flushleft}
    We can now check the following representation properties of $\mathcal{R}$ are satisfied. 
\end{flushleft}

\begin{proposition}\label{prop:S-T-rep-positivity-explicit}
For each $n\ge1$ and each $T\in M_n(M\bar\otimes M^{op})$, the operator
$\mathcal R$ defined in Definition~\ref{def:right-action-RS-explicit} satisfies:
\begin{enumerate}
  \item \emph{Linear}: $\mathcal{R}_{T+S}=\mathcal R_{T} +\mathcal R_S$
  \item \emph{Boundedness:} $
  \|\mathcal R_T\|\le \|T\|_{M_n(M\bar\otimes M^{op})}.$
  \newline
  Hence $\mathcal{R}_T\in B(\big(L^2(M){\bar\otimes}\,L^2(M^{op})\big)^{\oplus_n})\simeq M_n(B(L^2(M){\bar\otimes}\,L^2(M^{op}))$

  \item \emph{$^\ast$–structure:} $\mathcal R_{T^\ast}=(\mathcal R_T)^\ast$ where by definition $T^*$ is defined by $(T^*)_{ij}=T_{ji}^*$ for the usual $^*$-involution on $M\bar\otimes M^{op}$.

\item $\mathcal R_{(1\otimes 1)\otimes I_n}=\mathrm{id}_{\big(L^2(M){\bar\otimes}\,L^2(M^{op})\big)^{\oplus_n})}$.
  \item \emph{Multiplication rule:}
  for the usual matrix product (denoted also here as either without subscript or $\sharp$) in $M_n(M\bar\otimes M^{op})$ (i.e. usual matrix product and multiplication $\sharp$ in $M\bar\otimes M^{op}$),

\[
  \mathcal R_S\circ\mathcal R_T=\mathcal R_{T S}=\mathcal R_{T\sharp S}
  \qquad\text{for all }T,S\in M_n(M\bar\otimes M^{op}).
  \]
In particular if $T$ is invertible in $M_n(M\bar\otimes M^{op})$, then $\mathcal R_T$ is invertible and $(\mathcal R_T)^{-1}=\mathcal R_{T^{-1}}$.

    \item \emph{Positivity:}
  if $T\ge 0$ in $M_n(M\bar\otimes M^{op})$ (resp.\ $T>0$, i.e.\ $T\ge c\,((1\otimes 1)\otimes I_n)$ for some $c>0$), then 
  $\mathcal R_T\ge 0$ (resp.\ $\mathcal R_T\ge c\,\mathrm{id}$) as an operator on 
  $\big(L^2(M)\bar\otimes L^2(M^{op})\big)^{\oplus_n}$.
\end{enumerate}
Hence, $\mathcal{R}$ is a anti-representation or a representation of the opposite algebra $M_n(M\bar\otimes M^{op})^{op}$.
\end{proposition}

\begin{proof}
\begin{enumerate}

\item \emph{Boundedness.}
This is immediate (and already used by Dabrowski~\cite[Proof of Theorem~17]{Dab14} and \cite{Dab14v2}).

\item \emph{$^\ast$–structure.}
Let $\eta,\zeta\in\big(L^2(M){\bar\otimes}L^2(M^{op})\big)^{\oplus_n}$. Then
\begin{eqnarray}
\big\langle \mathcal R_T\eta, \zeta\big\rangle_{\big(L^2(M){\bar\otimes}\,L^2(M^{op})\big)^{\oplus_n}}
&=& \sum_{i=1}^n \big\langle (\mathcal R_T\eta)_i,\zeta_i\big\rangle_{L^2(M){\bar\otimes}\,L^2(M^{op})} \nonumber\\
&=& \sum_{i=1}^n \big\langle \sum_{j=1}^n \eta_j\sharp T_{ji}, \zeta_i\big\rangle_{L^2(M){\bar\otimes}\,L^2(M^{op})} \nonumber\\
&=& \sum_{i,j=1}^n \big\langle \eta_j\sharp T_{ji}, \zeta_i\big\rangle_{L^2(M){\bar\otimes}\,L^2(M^{op})} \nonumber\\
&=& \sum_{i,j=1}^n \big\langle \eta_j, \zeta_i\sharp T_{ji}^\ast\big\rangle_{L^2(M){\bar\otimes}\,L^2(M^{op})},
\end{eqnarray}
using the following trivial property: for all $\eta_1,\eta_2\in L^2(M){\bar\otimes}\,L^2(M^{op})$ and all $u\in M\bar\otimes M^{op}$,
one has $
\big\langle \eta_1\sharp u, \eta_2\big\rangle_{L^2(M){\bar\otimes}\,L^2(M^{op})}
=
\big\langle \eta_1, \eta_2\sharp u^\ast\big\rangle_{L^2(M){\bar\otimes}\,L^2(M^{op})},
$
where again $(a\otimes b)^\ast:=a^\ast\otimes b^\ast$ on simple tensors.
\newline
But for each $j=1,\ldots,n$,

\[
(\mathcal R_{T^\ast}\zeta)_j
= \sum_{i=1}^n \zeta_i\sharp (T^\ast)_{ij}
= \sum_{i=1}^n \zeta_i\sharp T_{ji}^\ast,
\]

so, that

\[
\big\langle \mathcal R_T\eta,\zeta\big\rangle_{\big(L^2(M){\bar\otimes}\,L^2(M^{op})\big)^{\oplus_n}}
= \sum_{j=1}^n \big\langle \eta_j, (\mathcal R_{T^\ast}\zeta)_j\big\rangle_{L^2(M){\bar\otimes}\,L^2(M^{op})}
= \big\langle \eta, \mathcal R_{T^\ast}\zeta\big\rangle_{\big(L^2(M){\bar\otimes}\,L^2(M^{op})\big)^{\oplus_n}}.
\]

Thus $\mathcal R_{T^\ast}=(\mathcal R_T)^\ast$.

\item \emph{Multiplicative structure.}
Let $T,S\in M_n(M\bar\otimes M^{op})$ and 
$\eta\in\big(L^2(M){\bar\otimes}L^2(M^{op})\big)^{\oplus_n}$.
Then

\[
(\mathcal R_T\eta)_j = \sum_{k=1}^n \eta_k\sharp T_{kj}.
\]

Hence, using associativity of $\sharp$,
\begin{align*}
(\mathcal R_S\mathcal R_T\eta)_i
&= \sum_{j=1}^n (\mathcal R_T\eta)_j\sharp S_{ji} \\
&= \sum_{j=1}^n \left(\sum_{k=1}^n \eta_k\sharp T_{kj}\right)\sharp S_{ji} \\
&= \sum_{k=1}^n \eta_k\sharp\Big(\sum_{j=1}^n T_{kj}\sharp S_{ji}\Big)=\sum_{k=1}^n \eta_k\sharp\Big(\sum_{j=1}^n T_{kj} S_{ji}\Big)
\end{align*}
The usual matrix product is

\[
(T\sharp S)_{ki} := \sum_{j=1}^n T_{kj} S_{ji}=\sum_{j=1}^n T_{kj}\sharp S_{ji},
\]
(recall that $\sharp$ is nothing but the canonical multiplication in $M\bar\otimes M^{op}$.)

so

\[
(\mathcal R_S\mathcal R_T\eta)_i
= \sum_{k=1}^n \eta_k\sharp (TS)_{ki}
= (\mathcal R_{TS}\,\eta)_i.
\]

Thus $\mathcal R_S\circ\mathcal R_T = \mathcal R_{TS}$.

\begin{flushleft}
    We can then easily see the inverse property by checking that $\mathcal R_{T^{-1}}$ is both a left and right inverse of $\mathcal{R}_T$.
\end{flushleft}
\item \emph{Positivity.}
Assume $T\ge0$ in $M_n(M\bar\otimes M^{op})$. Then $T=Q^\ast Q$ for some 
$Q\in M_n(M\bar\otimes M^{op})$. By the representation property and the 
$^\ast$–structure,

\[
\mathcal R_T = \mathcal R_{Q^\ast Q} = \mathcal R_Q\,\mathcal R_{Q^\ast}
= \mathcal R_Q\,(\mathcal R_Q)^\ast.
\]

Therefore, for any 
$\eta\in\big(L^2(M){\bar\otimes}\,L^2(M^{op})\big)^{\oplus_n}$,
\begin{align*}
\big\langle \mathcal R_T\eta,\eta\big\rangle_{\big(L^2(M){\bar\otimes}\,L^2(M^{op})\big)^{\oplus_n}}
&= \big\langle \mathcal R_Q(\mathcal R_Q)^\ast\eta,\eta\big\rangle_{\big(L^2(M){\bar\otimes}\,L^2(M^{op})\big)^{\oplus_n}} \\
&= \big\langle (\mathcal R_Q)^\ast\eta,(\mathcal R_Q)^\ast\eta\big\rangle_{\big(L^2(M){\bar\otimes}\,L^2(M^{op})\big)^{\oplus_n}} \\
&\ge 0.
\end{align*}
which yields the positivity statement.
\bigbreak 
If $T\geq c.(1\otimes 1)\otimes I_n$, in particular $T= c.(1\otimes 1)\otimes I_n+Q^*Q$ for some $Q\in M_n(M\bar\otimes M^{op})$, so that $\mathcal R_T=c\:\mathrm{id}+\mathcal{R}_Q(\mathcal{R}_Q)^*$, and the conclusion is immediate.
\end{enumerate}
\end{proof}
\qed

\begin{flushleft}
In particular, this formalism is necessary as the almost--commutation relation on $\Dom(\Delta^{3/2})$:

\[
\Dom(\Delta^{3/2})\subset \Dom\!\left(\overline{\Delta \otimes\mathrm{id}+\mathrm{id}\otimes \Delta }\circ\bar{\partial}\right),
\]

and for any $x\in \Dom(\Delta^{3/2})$,
\begin{align*}
\bar{\partial}_{i}\Delta(x)
&= \Delta^{\otimes}\bar{\partial}_i(x)
+\sum_{j=1}^N\bar{\partial}_{j}(x)\,\sharp\big(\bar{\partial}_{i}\partial_{j}^{*}(1\otimes 1)\big).
\end{align*}
Here the correction matrix $\big(\bar\partial_j(\partial_{i}^{*}(1\otimes 1))\big)_{ij}$ acts via $\mathcal{R}_{(\cdot)}$ through the right multiplication above done in a reversed order (note also the reversed index convention we adopt for the non-commutative Jacobian). Since we prove in Proposition~\ref{prop:S-T-rep-positivity-explicit} that any $C^\ast$--positivity condition on this matrix (such as the curvature--dimension condition) transfers directly to an operator inequality for $\mathcal{R}_T$, we are now in a good position.
\end{flushleft}
\begin{flushleft}
    Let us also remark now that under Lipschitz conjugate variables, the non-commutative Jacobian of such conjugates variables enjoys nice symmetry properties which will make it self-adjoint and fixed for the involution on $M\bar\otimes M^{op}$ corresponding to adjointness of Hilbert-Schmidt operators as recalled in Remark~\ref{hs}.
\end{flushleft}
\begin{lemma}[Symmetries of the Jacobian]\label{second}
Suppose $M=(W^\ast(X_1,\ldots,X_n),\tau)$ satisfies the \emph{Lipschitz conjugate variable condition}. Then:
\begin{enumerate}
    \item \emph{Schwarz integrability} (Dabrowski \cite[Lemma~36]{Dab14}):

\[
      \bar{\partial}_j\xi_i=\sigma(\bar{\partial}_i \xi_j),
    \]

    where the flip automorphism is given on simple tensors by $\sigma(a\otimes b)=b\otimes a$, $a\in M$, $b\in M^{op}$, and extends to an isometric involution on $L^2(M){\bar\otimes}\,L^2(M^{op})$.
    
    \item \emph{Self--adjointness of the Jacobian:} setting

\[
      \mathscr{J}\xi=(\bar{\partial}_j\xi_i)_{ij}\in M_n(M\bar{\otimes}M^{op}),
    \]

    one has for all $i,j$,

\[
      (\mathscr{J}\xi)_{ij}^\ast=(\mathscr{J}\xi)_{ji},
    \]

    where $(a\otimes b)^\ast=a^\ast\otimes b^\ast$ on simple tensors.
    
    \item \emph{HS fixed--point:}

\[
      ((\mathscr{J}\xi)_{i,j})^{\dagger}=(\mathscr{J}\xi)_{i,j}.
    \]

    \item If $\mathscr{J}\xi$ is invertible then $(\mathscr{J}\xi)^{-1}$ is self-adjoint.
\end{enumerate}
\end{lemma}

\begin{proof}
\begin{enumerate}
    \item The first point is \cite[Lemma~36]{Dab14}.
    
    \item For the invariance under the Hilbert-Schmidt involution $^\dagger$: since $\xi_j \in \Dom(\bar\partial)$ for each $j=1,\ldots,n$, and the free difference quotients are real derivations, we have

\[
      \bar\partial_j(\xi_i^\ast)=(\bar\partial_j\xi_i)^{\dagger}.
    \]

    As the conjugate variables are self--adjoint (cf.\ Proposition~\ref{self}), $\xi_i^\ast=\xi_i$, hence

\[
      (\bar\partial_j\xi_i)^{\dagger}=\bar\partial_j\xi_i.
    \]

    \item For self--adjointness, combine Schwarz integrability with the previous point and that $\sigma$ in an involutive automorphism:

\[
      (\bar\partial_j\xi_i)^\ast=\sigma\big((\bar\partial_j\xi_i)^{\dagger}\big)
      =\sigma(\bar\partial_j\xi_i)=\bar\partial_i\xi_j,
    \]

    yielding $(\mathscr{J}\xi)_{i,j}^\ast=(\mathscr{J}\xi)_{j,i}$.
    \item The self-adjointness of the inverse follows now as a consequence of the self-adjoitness of $\mathscr{J}\xi$.
\end{enumerate}
This completes the proof.
\end{proof}

\qed
\begin{flushleft}
We can now state one of our main theorems which gives a precise control via an Hessian term in the free Poincar\'e inequality under a convexity and invertibility assumption and henceforth of the Poincar\'e constant under a curvature criterion: in this case the free Poincar\'e constant is less than or equal to the inverse of the convexity constant, in perfect analogy to the classical case.
\end{flushleft}
\begin{theorem}[Free Brascamp--Lieb--Poincar\'e inequality]\label{thm:BLP-Lipschitz}
Let $(M,\tau)$ be a tracial von Neumann algebra generated by a self-adjoint $n$-tuple
$X=(X_1,\ldots,X_n)$, and assume the tuple admits Lipschitz conjugate variables in the
sense of Definition~\ref{def6}.  Define the non-commutative Jacobian of the conjugate
variables by

\[
\mathscr J\xi := \big(\bar\partial_j \xi_i\big)_{i,j=1}^n \;\in\; M_n(M\,\bar\otimes\,M^{op}).
\]

Assume that $\mathscr J\xi$ is bounded, invertible and positive in the $C^*$-algebra
$M_n(M\bar\otimes M^{op})$.  In particular, since $\mathscr J\xi>0$ and invertible,
its spectrum is contained in a compact subset of $[c,\infty)$ for some $c>0$, and hence
\begin{equation}\label{eq:CD-Cstar}
\mathscr J\xi \;\ge\; c\,(1\otimes 1)\otimes I_n
\qquad\text{in } M_n(M\bar\otimes M^{op}),
\end{equation}
i.e.\ the curvature--dimension condition $CD(c,\infty)$ holds in the $C^*$-sense.

Then for every $Y\in \Dom(\bar\partial)=\bigcap_{i=1}^n \Dom(\bar\partial_i)$ one has

\[
\big\|Y-\tau(Y).1\big\|_{L^2(M)}^2
\;\le\;
\sum_{i=1}^n\sum_{j=1}^n
\Big\langle\bar\partial_j Y\sharp\big((\mathscr J\xi)^{-1}\big)_{ji},\,
\bar\partial_i Y\Big\rangle_{L^2(M)\,\bar\otimes\,L^2(M^{op})}.
\]

Moreover, under the curvature--dimension condition \eqref{eq:CD-Cstar}, functional
calculus gives $(\mathscr J\xi)^{-1}\le \frac{1}{c}\,(1\otimes 1)\otimes I_n$, and
therefore

\[
\big\|Y-\tau(Y). 1\big\|_{L^2(M)}^2
\le
\frac{1}{c}\sum_{i=1}^n
\big\|\bar\partial_i Y\big\|_{L^2(M)\,\bar\otimes\,L^2(M^{op})}^2.
\]

\end{theorem}

\begin{proof}
By Lemma~\ref{second}, the conjugate variables satisfy Schwarz integrability,

\[
\bar{\partial}_j\xi_i=\sigma(\bar{\partial}_i \xi_j),
\]

and each entry $\bar\partial_j\xi_i$ is fixed by the dagger involution. In particular,
$\mathscr J\xi$ is $^\ast$–self-adjoint and strictly positive in
$M_n(M\bar\otimes M^{op})$, and its inverse $(\mathscr{J}\xi)^{-1}$ is then also
$^\ast$–self-adjoint and strictly positive.

\medskip
Recall that the Laplacian $
\Delta := \sum_{i=1}^n \partial_i^\ast \bar\partial_i
$
is a positive self–adjoint operator on $L^2(M)$ with $\ker(\Delta)=\C 1$.  Indeed, if $\Delta Y=0$, then the
Dirichlet form associated to $\Delta$ vanishes, i.e. 

\[
\mathcal E(Y)
:=
\sum_{i=1}^n \|\bar\partial_i Y\|_{L^2(M)\,\bar\otimes\,L^2(M^{op})}^2.
\]
By the free Poincaré inequality (which we recall is always satisfied), we get $\lVert Y-\tau(Y).1\rVert_{L^2(M)}$, so $Y=\tau(Y).1$ Hence $\Delta$ is injective on $L^2_0(M)$ which is a trivially a closed subspace of $L^2(M)$.

Moreover, we deduce
\begin{equation}
    \overline{\mathrm{Ran}(\Delta)}=\mathrm{Ker}(\Delta)^{\perp}=(\mathbb{C}1)^{\perp}=L_0^2(M):=\{Y\in L^2(M):\tau(Y)=0\}.
\end{equation}
and since we know $\tau\circ \phi_t=\tau$ where $\phi_t$ is the semigroup exponentiating $-\Delta$, an immediate differential argument gives $\tau\circ\Delta=0$, so that $L_0^2(M)$ is stable by $\Delta$ (it can even be seen more easily using duality and that $\bar\partial 1=0$).
\bigbreak

But we know a bit more since $\Delta$ is in fact coercive. This comes from free as a consequence of the free Poincar\'e inequality (spectral gap) which tell us that there exists a constant $C_P>0$ such that

\[
\|Y - \tau(Y).1\|_{L^2(M)}^2
\;\le\;
C_P\,\mathcal E(Y),
\qquad Y\in \Dom(\Delta).
\]

In particular, for $Z\in \Dom(\Delta)\cap L^2_0(M)$ (i.e.\ $\tau(Y)=0$), we have

\[
\|Y\|_{L^2(M)}^2
\;\le\;
C_P\,\langle \Delta Y, Y\rangle_{L^2(M)}.
\]

Equivalently,

\[
\langle \Delta Y, Y\rangle_{L^2(M)}
\;\ge\;
C_P^{-1}\,\|Y\|_{L^2(M)}^2,
\qquad Y\in \Dom(\Delta)\cap L^2_0(M).
\]

Thus, on the subspace $L^2_0(M)$, the operator
$\Delta$ is strictly positive and satisfies

\[
\Delta \;\ge\; C_P^{-1} \mathrm{id}
\quad\text{on }L^2_0(M).
\]

In particular, $\Delta$ is injective on $L^2_0(M)$ and has bounded inverse

\[
\Delta^{-1}:L^2_0(M)\to \Dom(\Delta)\cap L^2_0(M),
\]

with $\|\Delta^{-1}\|\le C_P$.

Hence we also have ${\mathrm{Ran}(\Delta)}$ is closed so $\overline{\mathrm{Ran}(\Delta)}={\mathrm{Ran}(\Delta)}=L_0^2(M)$.
\bigskip

Let now $Y\in \Dom(\bar\partial)=\Dom(\Delta^{1/2})$ be arbitrary. Replacing $Y$ by
$Y-\tau(Y).1$ if necessary, we may assume $Y$ is centered, so $Y\in L^2_0(M)$.
Since $\Delta$ is invertible on $L^2_0(M)$, there exists a unique $g := \Delta^{-1}Y \in \Dom(\Delta)\cap L^2_0(M)$, such that $Y = \Delta g$ (which provide thus a solution to a \emph{Poisson}-type equation).

\bigskip
Since $\Delta$ is a positive self–adjoint unbounded operator,  the
spectral theorem for unbounded self-adjoint operator yields a projection–valued measure $E(\cdot)$ on $\sigma(\Delta)\subset[0,\infty)$ such that

\[
\Delta=\int_{\sigma(\Delta)} \lambda\,dE(\lambda),
\]

and for every $Z\in L^2(M)$ the associated scalar spectral measure is

\[
\nu_Z(B):=\langle E(B)Z,Z\rangle_{L^2(M)},\qquad B\subset\sigma(\Delta)\ \text{Borel}.
\]

In particular,

\[
Z\in \Dom(\Delta^{1/2})
\ \Longleftrightarrow\
\int_{\sigma(\Delta)} \lambda\,d\nu_Z(\lambda)<\infty,\qquad
Z\in \Dom(\Delta^{3/2})
\ \Longleftrightarrow\
\int_{\sigma(\Delta)} \lambda^3\,d\nu_Z(\lambda)<\infty.
\]

From $Y\in \Dom(\bar\partial)=\Dom(\Delta^{1/2})$ and $Y=\Delta g$, we now show that
$g\in \Dom(\Delta^{3/2})$. By the spectral theorem,

\[
Y(\lambda)=\lambda\,g(\lambda)\qquad\text{for a.e.\ }\lambda\in\sigma(\Delta),
\]

and the scalar spectral measures satisfy

\[
d\nu_Y(\lambda)=|Y(\lambda)|^2\,d\mu(\lambda)
=\lambda^2\,|g(\lambda)|^2\,d\mu(\lambda)
=\lambda^2\,d\nu_g(\lambda).
\]

Hence

\[
Y\in \Dom(\Delta^{1/2})
\ \Longleftrightarrow\
\int_{\sigma(\Delta)} \lambda\,d\nu_Y(\lambda)
=\int_{\sigma(\Delta)} \lambda^3\,d\nu_g(\lambda)<\infty,
\]

which is exactly the spectral condition for $g\in \Dom(\Delta^{3/2})$. Thus

\[
Y\in \Dom(\Delta^{1/2})\ \text{ and }\ Y=\Delta g
\quad\Longleftrightarrow\quad
g\in \Dom(\Delta^{3/2}).
\]

\medskip

Under the Lipschitz conjugates assumption, the almost–commutation relation~\ref{in2deltaDelta}

\[
\bar{\partial}_i\,\Delta(x)
=
\Delta^{\otimes}\,\bar{\partial}_i(x)
+
\sum_{j=1}^n \bar{\partial}_j(x)\sharp \big(\bar{\partial}_i\xi_j\big),
\qquad x\in \Dom(\Delta^{3/2}),\ i=1,\dots,n.
\]

Since $g\in \Dom(\Delta^{3/2})$ and $Y=\Delta g$, we
can apply this with $x=g$ to obtain, for each $i$,

\[
\bar{\partial}_i Y
=
\bar{\partial}_i(\Delta g)
=
\Delta^{\otimes}(\bar{\partial}_i g)
+
\sum_{j=1}^n \bar{\partial}_j g\sharp\big(\bar{\partial}_i\xi_j\big).
\]

In vector form, this can be rewritten as

\[
\bar\partial Y
=
\big(\Delta^{\otimes}\otimes I_n\big)(\bar\partial g)
+
\mathcal{R}_{\mathscr{J}\xi}(\bar\partial g),
\]

where $\mathcal{R}_{\mathscr{J}\xi}$ is the right–leg multiplication operator associated to
$\mathscr{J}\xi$ (Definition~\ref{def:right-action-RS-explicit}).
\bigskip
Define now the operator

\[
\mathcal{A}
:=
\Delta^{\otimes}\otimes I_n+ \mathcal{R}_{\mathscr{J}\xi}
\]

on the domain

\[
\Dom(\mathcal{A})=\Dom(\Delta^{\otimes}\otimes I_n)=\big(\Dom(\Delta^{\otimes})\big)^{\oplus_n}.
\]

Since $\Delta$ generates a completely Dirichlet form (see Proposition~\ref{resolvent}). Moreover, the tensor extension $(\partial \otimes1)\oplus(1\otimes\partial)$ is again densely defined on $M_0\otimes M_0$ and closable, so it implies that
$\Delta^{\otimes}=\overline{\Delta\otimes1+1\otimes\Delta}$ is a densely defined closed self-adjoint positive operator also generating a completely
Dirichlet form. Hence, $\Delta^{\otimes}\otimes I_n\ge0$ on
$\big(L^2(M){\bar\otimes}\,L^2(M^{op})\big)^{\oplus_n}$.  On the other hand, $\mathcal{R}_{\mathscr{J}\xi}$ is bounded,
self–adjoint and strictly positive on
$\big(L^2(M){\bar\otimes}\,L^2(M^{op})\big)^{\oplus_n}$, since $\mathscr J\xi$ is
strictly positive and invertible in $M_n(M\bar\otimes M^{op})$ and
$\mathcal R_{\mathscr J\xi}$ its corresponding right–multiplication operator thus has by Proposition~\ref{prop:S-T-rep-positivity-explicit} the same properties: invertibility and coercivity. In
particular, there exists $c>0$ (take for example $c:=\|(\mathscr J\xi)^{-1}\|_{M_n(M\bar\otimes M^{op})}^{-1}$, since $\mathcal R_{\mathscr J\xi}^{-1}=\mathcal R_{(\mathscr J\xi)^{-1}}$ and $\|\mathcal R_{(\mathscr J\xi)^{-1}}\|\le \|(\mathscr J\xi)^{-1}\|_{M_n(M\bar\otimes M^{op})}$)  such that

\[
\langle \mathcal{R}_{\mathscr{J}\xi}u,u\rangle_{ \big(L^2(M){\bar\otimes}\,L^2(M^{op})\big)^{\oplus_n}}
\;\ge\;
c\,\|u\|_{ \big(L^2(M){\bar\otimes}\,L^2(M^{op})\big)^{\oplus_n}}^2,
\qquad u\in \big(L^2(M){\bar\otimes}\,L^2(M^{op})\big)^{\oplus_n}.
\]

Therefore, for all $u\in \Dom(\mathcal{A})$,

\begin{eqnarray}
\langle \mathcal{A}u,u\rangle_{ \big(L^2(M){\bar\otimes}\,L^2(M^{op})\big)^{\oplus_n}}
&=&
\langle (\Delta^{\otimes}\otimes I_n)u,u\rangle_{ \big(L^2(M){\bar\otimes}\,L^2(M^{op})\big)^{\oplus_n}}
+
\langle \mathcal{R}_{\mathscr{J}\xi}u,u\rangle_{ \big(L^2(M){\bar\otimes}\,L^2(M^{op})\big)^{\oplus_n}}\nonumber\\
&\ge&
c\,\|u\|_{\big(L^2(M){\bar\otimes}\,L^2(M^{op})\big)^{\oplus_n}}^2.
\end{eqnarray}

Thus $\mathcal{A}$ is self–adjoint, strictly positive and has bounded inverse

\[
\mathcal{A}^{-1}:\big(L^2(M){\bar\otimes}\,L^2(M^{op})\big)^{\oplus_n}\to
\big(L^2(M){\bar\otimes}\,L^2(M^{op})\big)^{\oplus_n}.
\]

By Löwner monotonicity (operator–monotonicity of $t\mapsto 1/t$ on $(0,\infty)$),
the inequality $\mathcal{A}\ge \mathcal{R}_{\mathscr{J}\xi}$ implies

\[
\mathcal{A}^{-1}\ \le\ \mathcal{R}_{\mathscr{J}\xi}^{-1}
\quad\text{on }\big(L^2(M){\bar\otimes}\,L^2(M^{op})\big)^{\oplus_n}.
\]

By Proposition~\ref{prop:S-T-rep-positivity-explicit}, we have the inverse rule:

\[
\mathcal{R}_{\mathscr{J}\xi}^{-1}
=
\mathcal{R}_{(\mathscr{J}\xi)^{-1}},
\]

and therefore

\[
\mathcal{A}^{-1}\ \le\ \mathcal{R}_{(\mathscr{J}\xi)^{-1}}
\quad\text{on }\big(L^2(M){\bar\otimes}\,L^2(M^{op})\big)^{\oplus_n}.
\]

From the almost–commutation identity we have

\[
\bar\partial Y=\mathcal{A}\,\bar\partial g,
\]

and since $\bar\partial g\in \Dom(A)$, we may apply $\mathcal{A}^{-1}$ to obtain

\[
\bar\partial g=\mathcal{A}^{-1}\bar\partial Y.
\]

\medskip
Using this relation and the integration–by–parts identity, we compute
\begin{align*}
\|Y\|_{L^2(M)}^2
&
=\big\langle \Delta g, Y\big\rangle_{L^2(M)}\\
&=\big\langle \bar\partial g, \bar\partial(\Delta g)\big\rangle_{(L^2(M)\,\bar\otimes\, L^2(M^{op}))^{\oplus_n}}\\
&=\big\langle \bar\partial g, \mathcal{A}\,\bar\partial g\big\rangle_{(L^2(M)\,\bar\otimes\, L^2(M^{op}))^{\oplus_n}}\\
&=\big\langle \mathcal{A}^{-1}\bar{\partial} Y, \bar{\partial} Y\big\rangle_{(L^2(M)\,\bar\otimes\, L^2(M^{op}))^{\oplus_n}}\\
&\le \big\langle \mathcal{R}_{(\mathscr{J}\xi)^{-1}}\bar{\partial} Y, \bar{\partial} Y\big\rangle_{(L^2(M)\,\bar\otimes\, L^2(M^{op}))^{\oplus_n}}.
\end{align*}
Expanding the last inequality gives, for each $i$,

\[
(\mathcal{R}_{(\mathscr{J}\xi)^{-1}}\bar\partial Y)_i
=
\sum_{j=1}^n \bar{\partial}_j Y\sharp\big((\mathscr{J}\xi)^{-1}\big)_{ji},
\]

so that we get

\[
\|Y\|_{L^2(M)}^2
\le
\sum_{i,j=1}^n
\Big\langle\,\bar{\partial}_j Y \sharp \big((\mathscr{J}\xi)^{-1}\big)_{ji} ,
\bar{\partial}_i Y\Big\rangle_{L^2(M)\,\bar\otimes\, L^2(M^{op})}.
\]

If, moreover, $\mathscr{J}\xi\ \ge\ c\,(1\otimes 1)\otimes I_n$ in the $C^\ast$–sense (note that invertibility and the positivity assumption on $\mathscr{J}\xi$ automatically gives such a $c:=\lVert (\mathscr{J}\xi)^{-1}\rVert_{M_n(M\bar\otimes M^{op})}^{-1}$),
then by positivity and functional calculus (the map $t\mapsto 1/t$ is
operator–monotone decreasing on $(c,\infty)$),

\[
(\mathscr{J}\xi)^{-1} \le \tfrac{1}{c}\,(1\otimes 1)\otimes I_n,
\]

and applying $\mathcal{R}$ yields

\[
\mathcal{R}_{(\mathscr{J}\xi)^{-1}}\ \le\ \tfrac{1}{c}\,\mathrm{id}
\quad\text{on }\big(L^2(M){\bar\otimes}\,L^2(M^{op})\big)^{\oplus_n}.
\]

Substituting this into the previous inequality gives the free Poincar\'e
(curvature–dimension) inequality

\[
\big\|Y-\tau(Y).1\big\|_{L^2(M)}^2
\ \le\
\frac{1}{c}\sum_{i=1}^n
\big\|\bar{\partial}_i Y\big\|^2_{L^2(M)\,\bar\otimes\, L^2(M^{op})},
\]

for all $Y\in \Dom(\bar\partial)$, which is the desired conclusion.
\end{proof}
\qed
\begin{remark}[Extension to real coercive coassociative derivations]\label{coass}
The Brascamp--Lieb--Poincar\'e inequality extends, with the same proof, to any
closable real coassociative derivation into the coarse correspondence with Lipschitz conjugates variables (w.r.t the derivation), in the
sense of \cite[Section~2.2, Assumption~2.1]{Dab14}, provided one has suitable
coercivity and boundedness assumptions.

Let $M=W^\ast(X_1,\ldots,X_n)$ and let

\[
\delta=(\delta_1,\ldots,\delta_n):\Dom(\delta)\longrightarrow
\big(L^2(M)\bar\otimes L^2(M^{op})\big)^{\oplus_n}
\]

be a densely defined derivation, real in the sense that
$(\delta_j(x))^{\dagger}=\delta_j(x^\ast)$ for all $x\in\Dom(\delta)$, and
coassociative in the sense that

\[
(\delta_j\otimes \mathrm{id})\circ \delta_i(x)
=
(\mathrm{id}\otimes \delta_i)\circ\delta_j(x),
\qquad x\in\Dom(\delta),\ 1\le i,j\le n.
\]

Let $\xi_j:=\delta_j^\ast(1\otimes 1)$ be the Lipschitz conjugate variables (so that
$\delta$ is closable), and define the closed Jacobian

\[
\mathscr J_\delta\xi
:=
\big(\,\overline{\delta}_j\,\xi_i\,\big)_{i,j=1}^n
\in M_n(M\bar\otimes M^{op}),
\]

together with the Laplacian
\(
\Delta_\delta:=\sum_{j=1}^n \delta_j^\ast\bar\delta_j
\)
on $L^2(M)$.

Assume in addition that:
\begin{enumerate}
  \item $\delta$ is closable, real and coassociative (so the almost--commutation
        identities and domain inclusions of Lemma~\ref{in2deltaDelta} hold);
  \item $\mathscr J_\delta\xi$ is bounded, positive and invertible in
        $M_n(M\bar\otimes M^{op})$ (equivalently, coercive on
        $\big(L^2(M)\bar\otimes L^2(M^{op})\big)^{\oplus_n}$);
  \item $\Delta_\delta$ is coercive on $L^2_0(M)$, so that
        $\ker(\Delta_\delta)=\C 1$.
\end{enumerate}
Then the same argument as in the free difference quotient case yields, for every
$Y\in\Dom(\overline\delta)$,

\[
\|Y-\tau(Y)1\|_{L^2(M)}^2
\le
\sum_{i,j=1}^n 
\big\langle\,\overline{\delta}_j Y\sharp
\big((\mathscr{J}_\delta\xi)^{-1}\big)_{ji},
\overline{\delta}_j Y\big\rangle_{L^2(M)\bar\otimes L^2(M^{op})}.
\]

 If
$\mathscr J_\delta\xi\ge c\,(1\otimes 1)\otimes I_n$ in the $C^\ast$--sense,
one recovers a Poincar\'e inequality with constant $1/c$.
\end{remark}

\begin{remark}[An example of application: free Gibbs states]
Let $V=V^\ast$ be a self-adjoint non‑commutative potential (polynomial, for
simplicity), and let $\tau_V$ be (assuming it does exist) a free Gibbs state generated by a bounded
$n$‑tuple $X$ satisfying the Schwinger--Dyson equations with conjugate variables
$\xi=\mathscr D V(X)$.  We say that $V$ satisfies $CD(c,\infty)$ if its
non‑commutative Hessian

\[
\mathscr J\mathscr D V=(\partial_j\mathscr D_i V)_{ij}
\]

obeys the algebraic positivity

\[
\mathscr J\mathscr D V \ \ge\ c\,(1\otimes 1)\otimes I_n
\qquad\text{in } M_n(\mathscr{P}\otimes \mathscr{P}^{op}),
\]
i.e.\ $\mathscr J\mathscr D V - c(1\otimes 1)\otimes I_n = Q^\ast Q$ for some
$Q$.
\par\vspace{0.5em}
equivalently on $M_n(\mathscr{P}\otimes \mathscr{P}^{op})$, since by Mai-Speicher-Weber~\cite{MSW} and Dabrowski~\cite[Lemma~37]{Dab14} the $X_i$ are algebraically free.
\par\vspace{0.5em}
Since $M_0$ is weakly dense in $M=W^\ast(X_1,\ldots,X_n)$ and the standard representation of
$M\bar\otimes M^{op}$ is faithful, this positivity extends to

\[
\mathscr J\mathscr D V(X) \ \ge\ c\,(1\otimes 1)\otimes I_n
\qquad\text{in } M_n(M\bar\otimes M^{op}).
\]
\end{remark}

\section{A free Obata's rigidity principle}
 \bigbreak
 This analytic framework in the previous section motivates the central question of the paper: \emph{can the rigidity phenomena of Cheng--Zhou \cite{ChengZhou} in the classical setting be translated into the free world?}
 \begin{flushleft}
We now prove an essential proposition to our approach: saturation of the free Poincaré inequality forces an eigenfunction relation for the Laplacian.     
 \end{flushleft}

\begin{proposition}\label{prop1}
In the setting of Theorem~\ref{thm:BLP-Lipschitz}, assume the curvature--dimension bound $CD(1,\infty)$:

\[
\mathscr{J}\xi\ \ge\ (1\otimes 1)\otimes I_n
\quad\text{in } M_n\!\big(M\,\bar\otimes\,M^{op}\big).
\]

Then for every centered $Y\in \Dom(\Delta^{1/2})=\Dom(\bar{\partial})$,

\[
\|Y\|_{L^2(M)}^2\ \le\ \sum_{i=1}^n \|\bar{\partial}_i Y\|_{L^2(M){\bar\otimes}\,L^2(M^{op})}^2.
\]

If $f\in \Dom(\Delta^{1/2})\cap L^2_0(M)$ saturates this free Poincaré inequality,

\[
\|f\|_{L^2(M)}^2\ =\ \sum_{i=1}^n \|\bar{\partial}_i f\|_{L^2(M){\bar\otimes}\,L^2(M^{op})}^2\ =:\ \mathcal{E}(f),
\]

then $f\in \Dom(\Delta)$ and $\Delta f=f$.
\end{proposition}

\begin{proof}
Recall that $\Dom(\Delta^{1/2})=\Dom(\bar{\partial})$ is the form domain, and fix 
$h\in \Dom(\Delta^{1/2})\cap L^2_0(M)$. 
For $\varepsilon\in\mathbb{R}$, the perturbations $f+\varepsilon h$ and 
$f+\varepsilon (ih)$ are centered (with $i^2=-1$), hence

\[
\|f+\varepsilon h\|_{L^2(M)}^2\ \le\ \mathcal{E}(f+\varepsilon h),
\qquad
\|f+\varepsilon(ih)\|_{L^2(M)}^2\ \le\ \mathcal{E}(f+\varepsilon(ih)),
\]

with equality at $\varepsilon=0$ by saturation. 
Expanding at $\varepsilon=0$ both left and right hand sides gives

\[
\|f+\varepsilon h\|_{L^2(M)}^2
= \|f\|_{L^2(M)}^2 + 2\varepsilon\,\Re\langle f,h\rangle_{L^2(M)} + O(\varepsilon^2),
\]

and

\[
\mathcal{E}(f+\varepsilon h)
= \mathcal{E}(f) + 2\varepsilon\,\sum_{i=1}^n 
\Re\langle \bar{\partial}_i f,\bar{\partial}_i h\rangle_{L^2(M){\bar\otimes}\,L^2(M^{op})} 
+ O(\varepsilon^2).
\]

Since $\|f\|_{L^2(M)}^2=\mathcal{E}(f)$, comparing the linear terms for 
$f+\varepsilon h$ gives

\[
\Re\langle f,h\rangle_{L^2(M)}
=
\sum_{i=1}^n 
\Re\langle \bar{\partial}_i f,\bar{\partial}_i h\rangle_{L^2(M)\bar\otimes L^2(M^{op})}.
\]

Applying the same argument to $f+\varepsilon(ih)$ (using conjugate linearity of the inner product
in the second variable), we get,

\[
\Re\langle f,ih\rangle_{L^2(M)}
=\Im\langle f,h\rangle_{L^2(M)},
\qquad
\Re\langle \bar{\partial}_i f,\bar{\partial}_i(ih)\rangle_{L^2(M)\bar\otimes L^2(M^{op})}
=\Im\langle \bar{\partial}_i f,\bar{\partial}_i h\rangle_{L^2(M)\bar\otimes L^2(M^{op})},
\]

and so, we obtain,

\[
\Im\langle f,h\rangle_{L^2(M)}
=
\sum_{i=1}^n 
\Im\langle \bar{\partial}_i f,\bar{\partial}_i h\rangle_{L^2(M)\bar\otimes L^2(M^{op})}.
\]

Hence, adding both real and imaginary parts, we get:

\[
\sum_{i=1}^n \langle \bar{\partial}_i f,\bar{\partial}_i h\rangle_{L^2(M){\bar\otimes}\,L^2(M^{op})}
=\langle \bar\partial f,\bar\partial h\rangle_{(L^2(M)\bar\otimes L^2(M^{op}))^{\oplus_n}}
=\langle f,h\rangle_{L^2(M)},
\qquad 
\forall h\in \Dom(\Delta^{1/2})\cap L^2_0(M).
\]

Therefore, we have, by the adjoint characterization,

\[
\bar\partial f\in \Dom(\partial^{\,*}) 
\quad\text{and}\quad 
\partial^{\,*}(\bar\partial f)=f.
\]

Hence, by definition,

\[
f\in \Dom(\Delta)
\quad\text{and}\quad 
\Delta f=\partial^{*}\bar\partial f=f.
\]

\end{proof}

\qed

\begin{remark}
The argument above can also be reformulated in a \emph{cleaner} variational way.  
Indeed, the Dirichlet form $
f\mapsto \mathcal{E}(f)$ is \emph{real} Fr\'echet differentiable on the Hilbert space $
\mathcal{H}:=\Dom(\Delta^{1/2})\cap L_0^2(M),
$ equipped with the graph norm 
\(
\|f\|_{\mathcal{H}}^2:=\|f\|_{L^2(M)}^2+\mathcal{E}(f).
\) 
Hence on $\mathcal{H}\setminus \left\{0\right\}$, the Rayleigh quotient

\[
R(g)=\frac{\mathcal{E}(f)}{\|f\|_{L^2(M)}^2},
\]

achieves its minimum at $f$.  
The corresponding Euler--Lagrange equation yields 
\(\Delta f = f\), 
recovering the conclusion of Proposition~\ref{prop1}.  
We omit the details, as this is simply a reformulation of the proof above.
\end{remark}

Let us just recall an easy lemma (whose proof is omitted) showing that such eigenfunctions have arbitrary fractional regularity.
\begin{lemma}\label{regul}
Let $\Delta$ be a selfadjoint, positive operator on $L^2(M)$.
Suppose $Y\in \Dom(\Delta)$ satisfies

\[
\Delta Y=\lambda\,Y \qquad \text{for some } \lambda\ge 0.
\]

Then for every $\alpha\ge 0$ one has

\[
Y\in \Dom(\Delta^\alpha)
\qquad\text{and}\qquad
\Delta^\alpha Y = \lambda^\alpha\,Y,
\]

where $\Delta^\alpha$ is defined by the Borel functional calculus for $\Delta$.
\end{lemma}

\begin{flushleft}
We now explain why equality and the eigenvalue relation pass to the real and imaginary parts of a complex extremizer. It will be useful later on in the paper, when we will want to handle only \emph{self-adjoint} elements.
\end{flushleft}

\begin{corollary}\label{cor1}
Under the curvature assumption $CD(1,\infty)$, if $f\in \Dom(\Delta)\cap L^2_0(M)$ satisfies $\Delta f=f$ and $\|f\|_{L^2(M)}^2=\mathcal{E}(f)$, then so do $\Re f$ and $\Im f$. In particular, without loss of generality one may assume $f=f^\ast:=Jf$.
\end{corollary}

\begin{proof}
Write $f=u+iv$ with $u=\Re f$, $v=\Im f$. Since $\tau(f)=0$, also $\tau(u)=\tau(v)=0$.  
We also have the $\ast$–stability of $\Dom(\Delta)$ and complex linearity, checked using the \emph{real} property of free difference quotients. Indeed, fix $x\in \Dom(\Delta)$, $y\in \Dom( \partial)$, then
\begin{eqnarray}
    \langle \Delta (x^*),y\rangle_{L^2(M)}&=&\sum_{i=1}^n \langle \bar \partial_i (x^*),\bar\partial_i y\rangle_{L^2(M)\bar\otimes L^2(M^{op})}\nonumber\\
    &=&\sum_{i=1}^n \langle (\bar \partial_i x)^{\dagger},\bar\partial_i y\rangle_{L^2(M)\bar\otimes L^2(M^{op})}\nonumber\\
    &=& \sum_{i=1}^n \langle (\bar \partial_i y)^{\dagger},\bar\partial_i x\rangle_{L^2(M)\bar\otimes L^2(M^{op})}\nonumber\\
    &=& \sum_{i=1}^n \langle \bar \partial_i (y^*),\bar\partial_i x\rangle_{L^2(M)\bar\otimes L^2(M^{op})}\nonumber\\
    &=& \langle y^*,\Delta x\rangle_{L^2(M)}=\langle (\Delta x)^*,y\rangle_{L^2(M)}
\end{eqnarray}
where we used for $a,b\in L^2(M)\bar\otimes L^2(M^{op})$ that $\langle a,b\rangle_{L^2(M)\bar\otimes L^2(M^{op})}=\langle b^{\dagger},a^{\dagger}\rangle_{L^2(M)\bar\otimes L^2(M^{op})}$. Hence $x^*\in \Dom(\Delta)$, and $\Delta(x^*)=(\Delta(x))^*:=J\Delta(x)$.
\bigbreak

Hence, we have $\Delta(f^*)=f^*$, and we get,

\[
\Delta u=\Re(\Delta f)=\Re(f)=u,\qquad
\Delta v=\Im(\Delta f)=\Im(f)=v.
\]

For the norms, traciality of $\tau$ gives

\[
\|f\|_{L^2(M)}^2=\|u\|_{L^2(M)}^2+\|v\|_{L^2(M)}^2.
\]

For the Dirichlet energies, the \emph{real} property for the free difference quotients and orthogonality yield

\[
\mathcal{E}(f)=\mathcal{E}(u)+\mathcal{E}(v).
\]

By hypothesis $\|f\|^2_{L^2(M)}=\mathcal{E}(f)$, hence

\[
\|u\|^2_{L^2(M)}+\|v\|^2_{L^2(M)}=\mathcal{E}(u)+\mathcal{E}(v).
\]

Therefore,

\[
(\|u\|^2_{L^2(M)}-\mathcal{E}(u))+(\|v\|^2_{L^2(M)}-\mathcal{E}(v))=0,
\]

By the Poincaré inequality under $CD(1,\infty)$, we have
$\|u\|_{L^2(M)}^2\le \mathcal{E}(u)$ and $\|v\|_{L^2(M)}^2\le \mathcal{E}(v)$, so each
difference $\|u\|_2^2-\mathcal{E}(u)$ and $\|v\|_2^2-\mathcal{E}(v)$ is
non-positive. Since their sum is zero, they must both vanish, and hence, forces equality termwise:

\[
\|u\|^2_{L^2(M)}=\mathcal{E}(u),\qquad \|v\|^2_{L^2(M)}=\mathcal{E}(v).
\]

\bigbreak
Thus $u,v$ are centered (eigen)functions saturating the free Poincaré inequality. Hence, we may assume without loss of generality $f=f^\ast$.
\end{proof}

\qed

\begin{flushleft}
    The following lemma show a nice \emph{tensoring} of the Dirichlet identity involving second-order free differences quotients, as a consequence of the Dirichlet identity at first order.
\end{flushleft}
\begin{proposition} \label{prop:tensor-dirichlet-Delta} In the setting of Theorem~\ref{thm:BLP-Lipschitz} with the generator \[ \Delta\ :=\ \sum_{j=1}^n \partial_j^{\,*}\,\bar{\partial}_j\ \ge 0, \qquad \Dom(\Delta^{1/2})=\Dom(\bar{\partial}), \] and the tensor extension of the closed Laplacian \[ \Delta^{\otimes}\ :=\ \overline{\Delta\otimes \mathrm{id}+\mathrm{id}\otimes \Delta}\ \ge 0 \quad\text{on }L^2(M){\bar\otimes}\,L^2(M^{op}). \] Then, for any $U,V\in \A\otimes \A^{op}$, \begin{align*} \big\langle \Delta^{\otimes} U,\ V\big\rangle_{L^2(M){\bar\otimes}\,L^2(M^{op})} &=\sum_{j=1}^n \Big( \big\langle (\partial_j\otimes \mathrm{id})U,\ (\partial_j\otimes \mathrm{id})V\big\rangle_{L^2(M)^{\otimes 3}}\\ &\hspace{2.7cm}+ \big\langle (\mathrm{id}\otimes \partial_j)U,\ (\mathrm{id}\otimes \partial_j)V\big\rangle_{L^2(M)^{\otimes 3}} \Big). \end{align*} \end{proposition} \begin{proof} By Voiculescu~\cite[Proposition~4.4]{V}, we have $M_0=\mathbb C \langle X_1,\ldots,X_n\rangle \subset \Dom(\Delta)$. Hence, on simple tensors $a\otimes b\in\A\otimes \A^{op}$, \[ \Delta^{\otimes}(a\otimes b)=\Delta a\otimes b\;+\;a\otimes \Delta b. \] Using the Dirichlet identity: \[ \langle \Delta a,\,c\rangle_{L^2(M)}=\sum_{j=1}^n \big\langle \partial_j a,\,\partial_j c\big\rangle_{L^2(M){\bar\otimes}\,L^2(M^{op})}. \] Take $U=a\otimes b$, $V=c\otimes d$. For the first leg, \begin{align*} \big\langle (\Delta a)\otimes b,\,c\otimes d\big\rangle_{L^2(M){\bar\otimes}\, L^2(M^{op})} &=\tau(c^\ast \Delta a)\,\tau(b d^\ast)\\ &=\sum_{j=1}^n \big\langle \partial_j a,\,\partial_j c\big\rangle_{L^2(M){\bar\otimes}\,L^2(M^{op})}\,\tau(b d^\ast)\\ &=\sum_{j=1}^n \big\langle (\partial_j a)\otimes b,\,(\partial_j c)\otimes d\big\rangle_{L^2(M)^{\otimes 3}}\\ &=\sum_{j=1}^n \big\langle (\partial_j\otimes \mathrm{id})U,\,(\partial_j\otimes \mathrm{id})V\big\rangle_{L^2(M)^{\otimes 3}}. \end{align*} For the second leg, \[ \big\langle a\otimes \Delta b,\,c\otimes d\big\rangle_{L^2(M){\bar\otimes}\, L^2(M^{op})} =\sum_{j=1}^n \big\langle (\mathrm{id}\otimes \partial_j)U,\,(\mathrm{id}\otimes \partial_j)V\big\rangle_{L^2(M)^{\otimes 3}}. \] Summing both terms, we then extended it by linearity to $\A\otimes \A^{op}$. \end{proof}
\qed 
\begin{flushleft}
    The following lemma shows that $\Dom(\Delta)$ embeds into a natural second--order Sobolev--type
domain $D\big(\overline{\partial\otimes \mathrm{id}\oplus\mathrm{id}\otimes \partial}\circ \bar\partial\big)$. This is in analogy with what happens in Malliavin (and free Malliavin) calculus~\cite{BS,Di2} where we have
the classical identity $\Dom(\mathcal{L})=\mathbb D^{2,2}$ for the Ornstein--Uhlenbeck operator $\mathcal{L}$.
Here we provide one inclusion, leaving the other one open and not investigated here and which is more or less equivalent to obtaining free variant of the so-called \emph{Meyer}'s inequalities (see e.g.  Nualart~\cite[Section~1.5]{Nualart} for details in the classical case).

\end{flushleft}

\begin{lemma}\label{lma5}
Assume $(M,\tau)$ satisfies the Lipschitz conjugate variables in the sense of Definition~\ref{def6}.

If $x\in \Dom(\Delta)$, then

\[
x\in \Dom(\overline{\partial\otimes \mathrm{id}\oplus\mathrm{id}\otimes \partial}\circ \bar\partial),
\qquad\text{equivalently}\qquad x\in \Dom(\Delta^{\otimes}\circ \bar\partial).
\]

\end{lemma}

\begin{proof}
Let $\xi_j=\partial_j^{\,*}(1\otimes 1)\in M$ be the Lipschitz conjugate variables.  
Define the Jacobian

\[
\mathscr{J}\xi=\big(\bar\partial_i\xi_j\big)_{i,j=1}^n\in M_n(M\bar\otimes M^{op}),
\]

and for $\eta=(\eta_1,\ldots,\eta_n)\in (L^2(M){\bar\otimes}\,L^2(M^{op}))^{\oplus_n}$, set
$\mathcal{R}_{\mathscr{J}\xi}\in B\big((L^2(M){\bar\otimes}\,L^2(M^{op}))^{\oplus_n}\big)$ by

\[
\big(\mathcal{R}_{\mathscr{J}\xi}\eta\big)_i=\sum_{j=1}^n \eta_j \sharp \bar\partial_i\xi_j.
\]

As recalled earlier (see Proposition~\ref{prop:S-T-rep-positivity-explicit}), $\mathcal{R}_{\mathscr{J}\xi}$ is bounded and

\[
\|\mathcal{R}_{\mathscr{J}\xi}\|\le \|\mathscr{J}\xi\|_{M_n(M\bar\otimes M^{op})}.
\]

\smallskip

Set the resolvent maps

\[
\eta_\alpha := \alpha(\alpha+\Delta)^{-1},
\qquad
{\eta}_\alpha^{\otimes} := \alpha(\alpha+\Delta^{\otimes})^{-1},
\]

where $\Delta^{\otimes}=\overline{\Delta\otimes\mathrm{id}+\mathrm{id}\otimes\Delta}$ on
$L^2(M)\,\bar\otimes\,L^2(M^{op})$. Then $\eta_\alpha$ and $\eta_\alpha^{\otimes}$ are unital, completely
positive, $L^2$–contractions, and

\[
\eta_\alpha\to\mathrm{id}\quad\text{on }L^2(M),
\qquad
\eta_\alpha^{\otimes}\to\mathrm{id}\quad\text{on }L^2(M)\,\bar\otimes\,L^2(M^{op}),
\]
strongly (terminology coming from seeing them respectively as operators on $B(L^2(M))$ and $B(L^2(M)\bar\otimes L^2(M^{op}))$) as $\alpha\to\infty$.
\bigbreak
We then extend $\eta_\alpha^{\otimes}$ diagonally to 
$(L^2(M)\,\bar\otimes\,L^2(M^{op}))^{\oplus_n}$, we then denote
$(\eta_\alpha^{\otimes})^{(n)}:=\eta_\alpha^{\otimes}\otimes I_n$ for this extension.

\bigbreak
Considering now $\eta_\alpha (x)$ for $x\in \Dom(\Delta^{1/2})=\Dom(\bar\partial)$.
Then immediate computations from spectral calculus gives $\eta_\alpha(x)\in \Dom(\Delta^{3/2})\subset \Dom(\bar\partial\circ\Delta)$, with moreover $\eta_\alpha(x)\to x$ in $L^2(M)$,
\bigbreak
In the same vein, we have immediately from spectral calculus, if moreover $x\in \Dom(\Delta)$,

\[
\Delta \eta_\alpha(x)
=
\alpha(x-\eta_\alpha(x))
\longrightarrow
\Delta x
\quad\text{in }L^2(M),
\]
\bigbreak
We now check convergence of $\bar\partial \eta_\alpha(x)\to \bar\partial x$
in $\big(L^2(M)\,\bar\otimes\,L^2(M^{op})\big)^{\oplus_n}$. 

The argument is a bit more involved and goes as follow. We invoke now a resolvent variant of Dabrowski's almost commutation formula which is  obtained in the following (see e.g. \cite[Lemma~2]{Dab14}), in our notation, we have for $x\in \Dom(\bar\partial)$:

\[
\bar\partial\,\eta_\alpha(x)
=
(\eta_\alpha^{\otimes})^{(n)}\,\bar\partial (x)
+
\tilde{\mathcal R}_\alpha(x),
\]

where $\tilde{\mathcal R}_\alpha: L^2(M)\to (L^2(M){\bar\otimes}\,L^2(M^{op}))^{\oplus_n}$ is a bounded operator with

\[
\tilde{\mathcal R}_\alpha
=
\frac{1}{\alpha}\,(\eta_\alpha^{\otimes})^{(n)}\,\mathcal R_{\mathscr J\xi}\big(\bar\partial\,\eta_\alpha\big), \quad\text{on}\: \Dom(\bar\partial)
\]

Indeed, for each $i=1,\ldots,n$, starting from the almost-commutation (since $\eta_{\alpha}(x)\in \Dom(\Delta^{3/2})\subset \Dom(\bar\partial \circ \Delta)$,
\[
\bar\partial_i\circ\Delta \eta_{\alpha}(x)
=
\Delta^{\otimes}\circ\bar\partial_i\eta_{\alpha}(x)
+(R_{\mathscr J\xi}\big(\bar\partial\,\eta_\alpha(x)\big))_i
\]
and applying $\eta_{\alpha}^{\otimes}$, we get

\[
\eta_\alpha^{\otimes}\bar\partial_i\circ \Delta \eta_\alpha(x)
=
\eta_\alpha^{\otimes}\circ \Delta^{\otimes}\,\bar\partial_i(\eta_{\alpha}(x))
+\eta_\alpha^{\otimes}\big(\big(\mathcal R_{\mathscr J\xi}(\bar\partial\,\eta_\alpha(x))\big)_i\big),
\]
now dividing by $\alpha>0$ and using the identities $\Delta\eta_\alpha=\alpha(\mathrm{id}-\eta_{\alpha})$ and its tensor resolvent $\Delta^{\otimes}\eta_\alpha^{\otimes}=\alpha(\mathrm{id}-\eta_{\alpha}^{\otimes})$,
\[
\eta_\alpha^{\otimes}\bar\partial_i(x)-\eta_\alpha^{\otimes}\bar\partial_i(\eta_\alpha(x))
=
\bar\partial_i\eta_\alpha^{\otimes}(x)-\eta_\alpha^{\otimes}\bar\partial_i(\eta_\alpha(x))
+\frac{1}{\alpha}\eta_\alpha^{\otimes}\big(\big(\mathcal R_{\mathscr J\xi}(\bar\partial\,\eta_\alpha(x))\big)_i\big),
\]
canceling identical terms in both sides, we immediately arrive at the conclusion for each coordinate and we write immediately in vector form using the diagonal extension of $\eta_\alpha^{\otimes}$ that we denoted above as $(\eta_\alpha^{\otimes})^{(n)}$.

\bigskip
Since $\eta_\alpha^{\otimes}$ is a contraction and $\mathcal R_{\mathscr J\xi}$ also, up to the
constant $C=\|\mathscr{J}\xi\|_{M_n(M\bar\otimes M^{op})}$, we have therefore

\[
\|\tilde{\mathcal R}_\alpha(x)\|_{\big(L^2(M)\,\bar\otimes\,L^2(M^{op})\big)^{\oplus_n}}
\le
\frac{C}{\alpha}\,\|\bar\partial\,\eta_\alpha(x)\|_{\big(L^2(M)\,\bar\otimes\,L^2(M^{op})\big)^{\oplus_n}}.
\]

Using the Dirichlet identity for $\Delta$ and the resolvent identity,

\[
\|\bar\partial\,\eta_\alpha(x)\|_{L^2(M)\bar\otimes L^2(M^{op})}^2
=
\langle \Delta\eta_\alpha(x),\,\eta_\alpha(x)\rangle_{L^2(M)}
=
\alpha\,\langle x-\eta_\alpha(x),\,\eta_\alpha(x)\rangle_{L^2(M)},
\]

we obtain

\[
\|\bar\partial\,\eta_\alpha(x)\|_{\big(L^2(M)\,\bar\otimes\,L^2(M^{op})\big)^{\oplus_n}}
\le
\alpha^{1/2}\,\|x-\eta_\alpha(x)\|_{L^2(M)}^{1/2}\,\|x\|_{L^2(M)}^{1/2}.
\]

Hence

\[
\frac{1}{\alpha}\,\|\bar\partial\,\eta_\alpha(x)\|_{\big(L^2(M)\,\bar\otimes\,L^2(M^{op})\big)^{\oplus_n}}
\le
\alpha^{-1/2}\,\|x-\eta_\alpha(x)\|_{L^2(M)}^{1/2}\,\|x\|_{L^2(M)}^{1/2}
\longrightarrow 0,
\]

since $\eta_\alpha(x)\to x$ in $L^2(M)$. It follows that

\[
\|\tilde{\mathcal R}_\alpha(x)\|_{\big(L^2(M)\,\bar\otimes\,L^2(M^{op})\big)^{\oplus_n}}
\longrightarrow 0.
\]

On the other hand, $\eta_\alpha^{\otimes}\to\mathrm{id}$ strongly on the tensor space, so its tensor extension also and then $(\eta_\alpha^{\otimes})^{(n)}\bar\partial x\to\bar\partial x$ in
$\big(L^2(M)\,\bar\otimes\,L^2(M^{op})\big)^{\oplus_n}$. Therefore

\[
\bar\partial \eta_{\alpha}(x)
=
(\eta_\alpha^{\otimes})^{(n)}\bar\partial (x) + \tilde{\mathcal R}_\alpha(x)
\longrightarrow
\bar\partial x
\quad\text{in }\big(L^2(M)\,\bar\otimes\,L^2(M^{op})\big)^{\oplus_n}.
\]

\bigskip

Finally, using the trivial bound (coming again from contractivity) $\|x-\eta_\alpha(x)\|_{L^2(M)}\le 2\|x\|_{L^2(M)}$, the Dirichlet identity gives the uniform estimate

\[
\|\bar\partial\,\eta_\alpha(x)\|_{\big(L^2(M)\,\bar\otimes\,L^2(M^{op})\big)^{\oplus_n}}^2 \le 2\alpha\,\|x\|_{L^2(M)}^2,
\qquad x\in \Dom(\bar\partial),
\]

and therefore

\[
\|\tilde{\mathcal R}_\alpha(x)\|
\le
C\sqrt{\frac{2}{\alpha}}\,\|x\|_{L^2(M)}.
\]

Thus $\tilde{\mathcal R}_\alpha$ is bounded in $L^2(M)$ on the dense domain $\Dom(\bar\partial)=\Dom(\Delta^{1/2})$, by density and linearity it extends uniquely to a bounded operator

\[
\tilde{\mathcal R}_\alpha : L^2(M)\longrightarrow \big(L^2(M)\,\bar\otimes\,L^2(M^{op})\big)^{\oplus_n},
\qquad
\|\tilde{\mathcal R}_\alpha\|\le C\sqrt{\tfrac{2}{\alpha}}.
\]

\smallskip

We are thus in position to apply the almost–commutation formula, and for each $i$, we have

\[
\bar\partial_i\Delta(\eta_\alpha(x))
=
\Delta^{\otimes}\,\bar\partial_i(\eta_\alpha(x))
+\sum_{j=1}^n \bar\partial_j(\eta_\alpha(x))\sharp \bar\partial_i\xi_j,
\]

or in vector form,

\[
\bar\partial\Delta(\eta_\alpha(x))
=
\Delta^{\otimes}\,\bar\partial(\eta_\alpha(x))
+\mathcal{R}_{\mathscr{J}\xi}\big(\bar\partial(\eta_\alpha(x))\big).
\]

\smallskip

Using the Dirichlet identity for $\Delta$, we obtain
\begin{align}\label{deltares}
\|\Delta \eta_\alpha(x)\|_{L^2(M)}^2
&=
\sum_{i=1}^n \big\langle \bar\partial_i\Delta(\eta_\alpha(x)),\,\bar\partial_i(\eta_\alpha(x))\big\rangle_{L^2(M){\bar\otimes}\,L^2(M^{op})}\nonumber\\
&=
\sum_{i=1}^n \big\langle \Delta^{\otimes}\,\bar\partial_i(\eta_{\alpha}(x)),\,\bar\partial_i(\eta_{\alpha}(x))\big\rangle_{L^2(M){\bar\otimes}\,L^2(M^{op})}\nonumber\\
&\quad+
\sum_{i=1}^n \big\langle (\mathcal{R}_{\mathscr{J}\xi}\bar\partial(\eta_{\alpha}(x)))_i,\,\bar\partial_i(\eta_{\alpha}(x))\big\rangle_{L^2(M){\bar\otimes}\,L^2(M^{op})}\nonumber\\
&=
\sum_{i,j=1}^n \big\|\overline{\partial_j\otimes \mathrm{id}}\,\bar\partial_i(\eta_{\alpha}(x))\big\|_{L^3(M)}^2
+\sum_{i,j=1}^n \big\|\overline{\mathrm{id}\otimes \partial_j}\,\bar\partial_i(\eta_{\alpha}(x))\big\|_{L^3(M)}^2\nonumber\\
&\quad+
\sum_{i=1}^n \big\langle (\mathcal{R}_{\mathscr{J}\xi}\bar\partial(\eta_{\alpha}(x)))_i,\,\bar\partial_i(\eta_{\alpha}(x))\big\rangle_{L^2(M){\bar\otimes}\,L^2(M^{op})}.
\end{align}
obtained again, using

\[
\Delta^{\otimes}
=
\big((\partial\otimes \mathrm{id})\oplus(\mathrm{id}\otimes \partial)\big)^{*}
\big((\partial\otimes \mathrm{id})\oplus(\mathrm{id}\otimes \partial)\big),
\]
The curvature term is bounded by

\[
\left|\sum_{i=1}^n \big\langle (\mathcal{R}_{\mathscr{J}\xi}\bar\partial(\eta_{\alpha}(x)))_i,\,\bar\partial_i(\eta_{\alpha}(x))\big\rangle_{L^2(M){\bar\otimes}\,L^2(M^{op})}\right|
\le \|\mathcal{R}_{\mathscr{J}\xi}\|\,\|\bar\partial \eta_{\alpha}(x)\|_{(L^2(M){\bar\otimes}\,L^2(M^{op}))^{\oplus_n}}^2.
\]

\smallskip

As $\alpha\to\infty$ in $L^2$, since $\mathcal R_{\mathscr J\xi}$ is bounded,

\[
\Delta \eta_{\alpha}(x)\to \Delta x,\qquad
\bar\partial \eta_{\alpha}(x)\to \bar\partial x,\qquad
\mathcal{R}_{\mathscr{J}\xi}\bar\partial(\eta_{\alpha}(x))\to \mathcal{R}_{\mathscr{J}\xi}\bar\partial(x).
\]

Hence, both the left–hand side of Equation~\ref{deltares}, namely 
$\|\Delta \eta_{\alpha}(x)\|_{L^2(M)}^{\,2}$, and the second term on the 
right–hand side (in both the second and third lines of the equality) 
converge as $\alpha\to\infty$, implying the following term

\[
\sum_{i,j=1}^n 
\big\|\overline{\partial_j\otimes \mathrm{id}}\,\bar\partial_i(\eta_{\alpha}(x))\big\|_{L^3(M)}^{2}
\;+\;
\sum_{i,j=1}^n 
\big\|\overline{\mathrm{id}\otimes \partial_j}\,\bar\partial_i(\eta_{\alpha}(x))\big\|_{L^3(M)}^{2},
\]

must also converge, and in particular is uniformly bounded in $\alpha$.
\newline
Because $\bar\partial\eta_\alpha(x)\to \bar\partial x$ in 
$\big(L^2(M)\,\bar\otimes\,L^2(M^{op})\big)^{\oplus_n}$ and the operator 
$\overline{\partial\otimes \mathrm{id}\oplus \mathrm{id}\otimes\partial}$ 
is closed, this uniform boundedness implies that the sequence 
$\overline{\partial\otimes \mathrm{id}\oplus \mathrm{id}\otimes\partial}\,
\bar\partial(\eta_\alpha(x))$ is bounded in $(L^2(M)\bar\otimes L^2(M)\bar\otimes L^2(M))^{\oplus_{2n^2}}$.  
By Banach--Alaoglu, we may extract a weakly convergent subsequence; and since 
the norms in \eqref{deltares} also converge, this weak convergence is in fact in $L^2$
(strong).  
By closedness, the limit $\bar\partial x$ therefore lies in the domain of 
$\overline{\partial\otimes \mathrm{id}\oplus \mathrm{id}\otimes\partial}$, 
which shows exactly the condition:

\[
x\in D\big(\overline{\partial\otimes \mathrm{id}\oplus\mathrm{id}\otimes \partial}\circ\bar\partial\big),
\]

which is then equivalent to \ $x\in \Dom(\Delta^{\otimes}\circ \bar\partial)$.
\end{proof}
\qed

\begin{definition}
For $Y\in \Dom({\Delta^{\otimes}}\circ \bar\partial):=\Dom(\overline{\Delta\otimes \mathrm{id}+\mathrm{id}\otimes \Delta}\circ \bar\partial)$, define the \emph{second-gradient energy} (by analogy with the Hessian term $\lVert \Hess (f)\rVert_2^2$ which does appear in the second Bakry-Emery $\Gamma_2$ operator)

\[
\mathcal E_2(Y) := \sum_{i=1}^n \big\langle \Delta^{\otimes}(\bar\partial_i Y), \bar\partial_i Y\big\rangle_{L^2(M){\bar\otimes}\,L^2(M^{op})}.
\]

On $\Dom(\Delta$) (by the previous Lemma~\ref{lma5}), this equals

\[
\mathcal E_2(Y)
=
\sum_{i,j=1}^n
\Big(
\|\overline{\partial_j\otimes\mathrm{id}}\, \bar\partial_i Y\|_{L^2(M)^{\otimes 3}}^2
+
\|\overline{\mathrm{id}\otimes\partial_j}\,\bar \partial_i Y\|_{L^2(M)^{\otimes 3}}^2
\Big),
\]

and by closability extends to all $Y\in \Dom(\Delta^{\otimes}\circ \bar \partial)$.
Using the conjugate variables $\xi_j=\partial_j^{\,*}(1\otimes 1)$, define the \emph{conjugate curvature contraction}

\[
C_\xi(Y) := \sum_{i,j=1}^n \big\langle \bar\partial_j Y\# \bar\partial_i \xi_j, \bar\partial_i Y\big\rangle_{L^2(M){\bar\otimes} L^2(M^{op})}.
\]

\end{definition}

\begin{flushleft}
We now turn to a key identity, strongly reminiscent of the Bakry--\'Emery calculus. 
It reveals a precise relation between the \emph{tensor Dirichlet} form and the \emph{curvature contraction}, 
and will play a central role in our rigidity argument.
\end{flushleft}

\begin{lemma}\label{lem:one-line-Delta}
If $Y\in \Dom(\Delta)$ and $\Delta Y=Y$, then

\[
\mathcal E_2(Y)= \mathcal{E}(Y)- C_\xi(Y).
\]

\end{lemma}

\begin{proof}
Apply the almost-commutation relation to $Y$ (which lies in $\Dom(\Delta^{3/2})$ by Lemma~\ref{regul} under $\Delta Y=Y$):

\[
\bar\partial_i\,\Delta(Y)= \Delta^{\otimes}(\bar\partial_i Y) + \sum_{k=1}^n \bar\partial_j Y \# \bar\partial_i \xi_j.
\]

Since $\Delta Y=Y$,

\[
\bar\partial_i Y = \Delta^{\otimes}(\bar\partial_i Y) + \sum_{j=1}^n \bar\partial_j Y \# \bar\partial_i \xi_j.
\]

Take the inner product with $\bar\partial_i Y$ and sum over $i$:

\[
\sum_{i=1}^n \|\bar\partial_i Y\|^2_{L^2(M){\bar\otimes}\,L^2(M^{op})}
=\sum_{i=1}^n \langle \Delta^{\otimes}(\bar\partial_i Y),\bar\partial_i Y\rangle_{L^2(M){\bar\otimes}\,L^2(M^{op})}
+\sum_{i,j=1}^n \langle \bar\partial_j Y\# \bar\partial_i \xi_j,\bar\partial_i Y\rangle_{L^2(M){\bar\otimes}\,L^2(M^{op})},
\]

i.e., $\mathcal E_2(Y)=\mathcal{E}(Y)-C_\xi(Y)$.
\end{proof}
\qed
\begin{remark}
    We could derive also the Bracsamp-Lieb inequality using the almost-commutation as well as taking inner products on both sides, the conclusion follows again from using the Poisson equation and a Cauchy-Scharwtz inequality.
\end{remark}
\begin{proposition}[Rigidity for $\Delta$]\label{prop:Szero-Delta}
Assume the $CD(1,\infty)$ curvature criterion.
If $Y\in \Dom(\bar{\partial})\cap L^2_0(M)$ is centered and saturates the free Poincar\'e inequality: $
\|Y\|_{L^2(M)}^2=\mathcal{E}(Y)$, then

\[
\mathcal E_2(Y)=0.
\]

Assume moreover that

\[
Y\in D\!\big(\overline{{\partial}\otimes \mathrm{id}\;\oplus\;\mathrm{id}\otimes{\partial}}\circ\bar\partial \big).
\]

Then for all $i,j=1,\ldots,n$,

\[
\overline{\partial_j\otimes \mathrm{id}}\,\bar\partial_i Y=0,
\qquad
\overline{\mathrm{id}\otimes \partial_j}\,\bar\partial_i Y=0.
\]

\end{proposition}

\begin{proof}
By Lemma~\ref{lem:one-line-Delta}, we have the identity

\[
\mathcal E_2(Y)=\mathcal{E}(Y)-C_\xi(Y).
\]

Under the $CD(1,\infty)$ curvature criterion, the contraction term satisfies

\[
C_\xi(Y)\;\ge\;\mathcal{E}(Y),
\]

by the right-leg comparison estimate. Hence $\mathcal E_2(Y)\le 0$.

Since $\mathcal E_2(Y)$ is a quadratic form and therefore nonnegative, we conclude

\[
\mathcal E_2(Y)=0.
\]

\smallskip

Now assume

\[
Y\in D\!\big(\overline{{\partial}\otimes \mathrm{id}\;\oplus\;\mathrm{id}\otimes{\partial}}\circ\bar\partial \big).
\]

By definition of $S(Y)$ and its explicit expression

\[
 \mathcal E_2(Y)
=
\sum_{i,j=1}^n
\Big(
\|\,\overline{\partial_j\otimes\mathrm{id}}\,\bar\partial_i Y\|_{L^2(M)^{\otimes 3}}^2
+
\|\,\overline{\mathrm{id}\otimes\partial_j}\,\bar\partial_i Y\|_{L^2(M)^{\otimes 3}}^2
\Big),
\]

the equality $S(Y)=0$ forces each term to vanish. Therefore, for all $i,j$,

\[
\overline{\partial_j\otimes \mathrm{id}}\,\bar\partial_i Y=0,
\qquad
\overline{\mathrm{id}\otimes \partial_j}\,\bar\partial_i Y=0.
\]

The second identity is in fact immediate from the coassociativity of free difference quotients
(see Remark~\ref{coass}).
\end{proof}
\qed
\begin{flushleft}
Our next goal is to show that saturation of the free Poincaré inequality forces extremizers to be \emph{affine} in the variables \(X_1,\ldots,X_n\). This is exactly what happens in the canonical semicircular model, where the generator \(\Delta=-\mathcal{L}_{V_0}\), with \(V_0=\tfrac12\sum_{j=1}^n X_j^2\), is the Ornstein--Uhlenbeck operator (the minus sign follows the usual Malliavin–calculus convention). In that case \(\Delta X_i=X_i\), and the full spectral decomposition is given by the orthogonal basis of Chebyshev polynomials of the second kind (see Biane--Speicher~\cite{BS} or~\cite{Di2}). This explicit structure is fundamental, for instance, in free Malliavin calculus.

In the general setting, however, and contrary to~\cite[Proposition~3]{Di3}, where higher--order free Poincaré inequalities are available (for now) only in the semicircular case, the affine conclusion is not immediate: higher--order free Poincaré inequalities are not yet known in full generality and only provide first--order control. Establishing such inequalities remains an open and interesting problem.
\end{flushleft}

\begin{flushleft}
To bypass this difficulty, we use a key reduction step in the next lemma: we move the problem to the bimodule setting via a \emph{slicing} argument and conclude again using the free Poincaré inequality. This idea is loosely inspired by a technique of Dabrowski~\cite[Proof of Lemma~19(ii)]{Dab14}, where the boundedness of the composition $(\mathrm{id}\otimes\tau)\circ\bar\partial_i$ is used to show stability of $D(\bar\partial)$ under slicing by $\mathrm{id}\otimes \tau$. Here we extend this approach to a family of separating slice maps. Since we only work with norms and inner products, we avoid passing to the opposite module by fixing the canonical identification $L^2(M^{op})\simeq L^2(M)$.
\end{flushleft}

\begin{lemma}\label{lem:kernel-scalar-L2}
Let $M=(W^*(X_1,\ldots,X_n),\tau)$ be a tracial $W^\ast$--probability space. Let $\bar{\partial}:=(\bar{\partial}_1,\ldots,\bar{\partial}_n)$
denote the closed free difference quotients of the self-adjoint variables $X=(X_1,\ldots,X_n)$.

Let

\[
\bar\partial_j:\Dom(\bar\partial_j)\subset L^2(M)\ \longrightarrow\ L^2(M)\,\bar\otimes\,L^2(M).
\]

Since each $\bar\partial_j$ is closed, the maps

\[
\overline{\partial_j\otimes \mathrm{id}},\qquad \overline{\mathrm{id}\otimes \partial_j}
\]

are also densely defined (i.e. $\mathbb{C}\langle X_1,\ldots,X_n\rangle \otimes \mathbb{C}\langle X_1,\ldots,X_n\rangle$ being a core) closed operators on $L^2(M)\,\bar\otimes\,L^2(M)$, obtained as the closures on the polynomial core.

Suppose $U\in \Dom(\overline{\partial_j\otimes \mathrm{id}\oplus\mathrm{id}\otimes \partial_j})$ for all $j=1,\ldots,n$, and satisfies

\[
\overline{\partial_j\otimes \mathrm{id}}(U)=0
\quad\text{and}\quad
\overline{\mathrm{id}\otimes \partial_j}(U)=0.
\]

Then $U=c\,(1\otimes 1)$ for some $c\in\C$.
\end{lemma}

\begin{proof}
We divide the proof into five steps:
\begin{enumerate}

 \item Step 1: For $b\in L^2(M)$ define $\omega_b:L^2(M)\to\C$ by
$\omega_b(y):=\tau(y b)$.  
With the $L^2$ inner product $\langle y,z\rangle_{L^2(M)}=\tau(z^\ast y)$
and traciality, we have $\omega_b(y)=\langle y, b^\ast\rangle_{L^2(M)}$.
By Cauchy--Schwarz,

\[
  |\omega_b(y)|\ \le\ \|y\|_{L^2(M)}\,\|b\|_{L^2(M)},
\]

so $\omega_b$ is bounded with $\|\omega_b\|\le \|b\|_{L^2(M)}$.

Consequently, the slice map

\[
  T_b:=\mathrm{id}\otimes \omega_b:\ 
  L^2(M)\,\bar\otimes\,L^2(M)\ \longrightarrow\ L^2(M)
\]

is bounded, and $\|T_b\|\le \|b\|_{L^2(M)}$.

Let $M_0^{\otimes 2}:=\C\langle X_1,\ldots,X_n\rangle
\otimes \C\langle X_1,\ldots,X_n\rangle$ denote the algebraic tensor product of non-commutative polynomials in the $X_i$.
Since $\omega_b$ is linear and $\mathrm{id}\otimes\omega_b$ acts
componentwise, we have in particular that $T_b$ send tensor product of non-commutative polynomials onto non-commutative polynomials, i.e. $T_b(M_0^{\otimes 2})\subset M_0$.

  \item Step 2: On (algebraic) tensor product of non-commutative polynomials, i.e. for  $Q\in\mathbb{C}\langle X_1,\ldots,X_n\rangle\otimes\mathbb{C}\langle X_1,\ldots,X_n\rangle$, a direct computation shows

\[
  \partial_j\big(T_b Q\big) = (\mathrm{id}\otimes \mathrm{id}\otimes \omega_b)\big((\partial_j\otimes \mathrm{id})Q\big).
  \]

  Fix now $j=1,\ldots,n$ and  suppose $U\in \Dom(\overline{\partial_j\otimes \mathrm{id}})$.  We 
  choose a sequence $U_n\in  \mathbb C\langle X_1,\ldots,X_n\rangle\otimes \mathbb C\langle X_1,\ldots,X_n\rangle$ converging to $U\in \Dom(\overline{\partial_j\otimes \mathrm{id}})$, i.e. such that $U_n\to U$ in $L^2(M)\,\bar\otimes\,L^2(M)$ and  
  $(\partial_j\otimes \mathrm{id})U_n\to W:=\overline{\partial_j\otimes \mathrm{id}}(U)$ in $L^2(M)\,\bar\otimes\,L^2(M)\,\bar\otimes\,L^2(M)$.  
  Since $T_b$ and $\mathrm{id}\otimes\mathrm{id}\otimes \omega_b$ are bounded, we get

\[
  T_b U_n \to T_b U,\qquad
  \partial_j(T_bU_n)=(\mathrm{id}\otimes\mathrm{id}\otimes \omega_b)\big((\partial_j\otimes \mathrm{id})U_n\big)\to (\mathrm{id}\otimes\mathrm{id}\otimes \omega_b)(W).
  \]

  Using closability of $\bar\partial_j$, we conclude

\[
  T_b U\in \Dom(\bar\partial_j),\qquad
  \bar\partial_j(T_b U)\ =\ (\mathrm{id}\otimes\mathrm{id}\otimes \omega_b)\big((\overline{\partial_j\otimes \mathrm{id}})U\big).
  \]

  In particular, if $(\overline{\partial_j\otimes \mathrm{id}})U=0$, then $\bar\partial_j(T_b U)=0$ for all $j$.

  \item Step 3: Applying the free Poincaré inequality to $T_b U\in \Dom(\bar\partial)$ yields

\[
  \|T_b U-\tau(T_b U).1\|_{L^2(M)}^2\ \le\ C \sum_{j=1}^n \|\bar\partial_j(T_b U)\|_{L^2(M)\,\bar\otimes\,L^2(M)}^2\ =\ 0,
  \]

  hence

\[
  T_b U\ =\ \tau(T_b U).1\qquad\text{for all }b\in L^2(M).
  \]

    \item Step 4:
  By Step~3, for each $b\in L^2(M)$ there exists a scalar $\lambda(b):=\tau(T_bU)\in\C$ such that

\[
    T_b U = \lambda(b).1.
  \]

  We now interpret this condition via the Hilbert--Schmidt structure.
  With the canonical identification
$L^2(M)\bar\otimes L^2(M) \simeq HS(L^2(M))$
coming from an isometric isomorphism (as recalled in the Preliminaries~\ref{HS})
   $ \Psi: L^2(M)\bar\otimes L^2(M) \rightarrow HS(L^2(M))$, we let $T:=\Psi(U)\in HS(L^2(M))$ denote the Hilbert--Schmidt operator associated to $U$.

  For a simple tensor $a\otimes c\in M\otimes M$ and $b\in L^2(M)$, we have $T_b(a\otimes c)
    = a\,\omega_b(c)
    = a\,\tau(cb)
    = \Psi(a\otimes c)(b)$

  By linearity and continuity (both maps coincide on the algebraic tensor product and both operators are bounded), this identity extends to all $U\in L^2(M)\,\bar\otimes\,L^2(M)$, in particular:

\[
    T_b U = T(b)\qquad\text{for all }b\in L^2(M).
  \]

  Combining this with $T_b U=\lambda(b).1$ from above, we obtain
\[
    T(b) = \lambda(b).1\qquad\text{for all }b\in L^2(M),
  \]

  so that the Hilbert–Schmidt operator $T$
 has range contained in $\mathbb C1$
and therefore $T$ is a finite‑rank operator of rank one (unless trivial, i.e., equivalently $U=0$ and there is nothing to prove). Thus there exists a bounded linear functional $\varphi:L^2(M)\to\C$ such that $
    T(b) = \varphi(b).1$ {for all }$b\in L^2(M)$. Since $T\in HS(L^2(M))$, $\varphi$ is continuous, and by the Riesz representation theorem there exists $d\in L^2(M)$ such that

\[
    \varphi(b)=\langle b,d\rangle_{L^2(M)}=\tau(d^\ast b)\qquad\text{for all }b\in L^2(M).
  \]

  Setting $b_0:=d^\ast$, we get $\varphi(b)=\tau(b_0 b)$ and hence

\[
    T(b) = \tau(b_0 b).1\qquad\text{for all }b\in L^2(M).
  \]

  On the other hand, the tensor $1\otimes b_0$ corresponds under $\Psi$ to the operator

\[
    \Psi(1\otimes b_0)(b) = \tau(b_0 b).1,
  \]

  which coincides with $T(b)$ for all $b$. Therefore

\[
    T = \Psi(U) = \Psi(1\otimes b_0).
  \]

  Since $\Psi$ is an isometric isomorphism (injective in particular), it follows that

\[
    U = 1\otimes b_0.
  \]

  \item Step 5: From $\overline{\mathrm{id}\otimes \partial_j}(U)=0$ and $U=1\otimes b_0$, we obtain

\[
  1\otimes \bar\partial_j(b_0) =(\overline{\mathrm{id}\otimes \partial_j})(U)\ =\ 0,
  \]

  so $\bar\partial_j(b_0)=0$ for all $j$.  
  Applying the free Poincaré inequality to $b_0$ yields $b_0=\tau(b_0).1$, whence

\[
  U\ =\ \tau(b_0).(1\otimes 1),
  \]

  which is the desired conclusion.

\end{enumerate}
\end{proof}
\qed

\begin{flushleft}
    Let us prove now a nice consequence about such extremizers of the free Poincar\'e inequality under the Lipschitz condition and $CD(1,\infty)$ criterion: affine rigidity, which means that such extremizers are necessarily affine in the generators $X_1,\ldots,X_n$ which is a key to achieve our main theorem. 
\end{flushleft}

\begin{corollary}\label{affine}
Assume $M=(W^*(X_1,\ldots,X_n),\tau)$ admits \emph{Lipschitz conjugate variables} $(\xi_1,\dots,\xi_n)$ as in Definition~\ref{def6}, 
and that their Jacobian satisfies the $CD(1,\infty)$ criterion.

Let $f\in \Dom(\bar\partial)=\Dom(\Delta^{1/2})$ centered, and saturate the free Poincaré inequality (so that $\Delta f=f$). Then

\[
f=\sum_{j=1}^n c_j\,X_j+c_0. 1
\]

for some scalars $c_j\in\C$.
\end{corollary}

\begin{proof}
By Lemma~\ref{lem:one-line-Delta} and the curvature bound we obtain $S(f)=0$, hence

\[
\overline{\partial_j\otimes \mathrm{id}}\:\bar\partial_i f=0, \qquad \overline{\mathrm{id}\otimes \partial_j}\:\bar \partial_i f=0
\]

for all $i,j=1,\ldots,n$. 
\newline
Now, Lemma~\ref{lem:kernel-scalar-L2} then yields $\bar\partial_i f=c_i.(1\otimes 1)$. 

Define $g:=f-\sum_{j=1}^n c_j X_j$. Since $\partial_i X_j=\delta_{ik}(1\otimes 1)$, we have

\[
\bar\partial_i g=\bar\partial_i f-c_i\,\partial_i X_i=c_i(1\otimes 1)-c_i(1\otimes 1)=0
\]

for all $i$. By the free Poincar\'e inequality, $g=c_0.1$ for some $c_0\in\C$, which yield the desired conclusion.
\end{proof}
\qed
\begin{remark}
As noted in Proposition~\ref{prop5}, the second-order difference quotients $(\partial_i\otimes \mathrm{id})\circ \partial_j$ are densely defined and closable, since second-order conjugate variables exist. This means that we could also consider $\overline{(\partial_i\otimes \mathrm{id})\circ \partial_j}$ directly without distinction.
\end{remark}

\begin{lemma}
\label{lem:affine-conjugates-semicircle}
Under the assumptions of Corollary~\ref{affine}, assume in addition that $f$ is self-adjoint.
Set $X_j^\circ:=X_j-\tau(X_j).1$ and $\sigma:=\|f\|_{L^2(M)}$.
Then:
\begin{enumerate}
\item One has

\[
  f=\sum_{j=1}^n c_j\,X_j^\circ,
\]

with $c_j\in\R$ and $\sum_{j=1}^n c_j^2=\sigma^2$.

\item For every $g\in\A$,

\[
  \langle f,g\rangle_{L^2(M)}
  =\sum_{j=1}^n \langle \partial_j f,\partial_j g\rangle_{L^2(M)\bar\otimes\,L^2(M^{op})}
  =\sum_{j=1}^n c_j\,\langle 1\otimes 1,\partial_j g\rangle_{L^2(M)\,\bar\otimes\,L^2(M^{op})}
  =\sum_{j=1}^n c_j\,\langle \xi_j,g\rangle_{L^2(M)},
\]

hence $f=\sum_{j=1}^n c_j\,\xi_j$.

\item Define the derivation (into $L^2(M)\,\bar\otimes\,L^2(M^{op})$)

\[
  \delta_f := \sum_{j=1}^n \frac{c_j}{\sigma}\,\partial_j
\]

Then, it is a closable derivation with $\widehat{f}:=\frac {f}{\sigma}=\delta_f^\ast(1\otimes 1)
  =\sum_{j=1}^n \frac{c_j}{\sigma}\,\xi_j$ with moreover $\delta_f(\widehat f)=1\otimes 1$.

\item The element $\widehat{f}$ has  the standard semicircular distribution.
\end{enumerate}
\end{lemma}

\begin{proof}
\begin{enumerate}

  \item 
  From Corollary~\ref{affine}, for each $i=1,\ldots,n$ we have

\[
      \partial_i f = c_i.(1\otimes 1),
  \]

  with $c_i\in\C$. Since $\partial_i X_j^\circ=\delta_{ij}(1\otimes 1)$ and $f$ is centered,

\[
      0=\tau(f)=\sum_{j=1}^n c_j\tau(X_j)+c_0,
  \]

  so $c_0=-\sum_{j=1}^n c_j\tau(X_j)$ and therefore

\[
      f=\sum_{j=1}^n c_j\big(X_j-\tau(X_j).1\big)=\sum_{j=1}^n c_j X_j^\circ.
  \]

  Moreover,

\[
      \|f\|_{L^2(M)}^2
      =\sum_{j=1}^n \|\partial_j f\|_{L^2(M)\,\bar\otimes\,L^2(M^{op})}^2
      =\sum_{j=1}^n |c_j|^2,
  \]

  so $\sigma^2=\sum_{j=1}^n |c_j|^2$.  
  If $f=f^\ast$, then by the \emph{real} property of free difference quotients,

\[
      c_i.(1\otimes 1)=\partial_i f=\partial_i(f^\ast)
      =(\partial_i f)^\dagger=\overline{c_i}.(1\otimes 1),
  \]

  hence $c_i\in\R$ for all $i$.  
  This proves the first item.

  \item 
  Because $f$ saturates the free Poincaré inequality, we have $\Delta f=f$.  
  The Dirichlet identity gives, for any $g\in\A$,

\[
      \langle f,g\rangle_{L^2(M)}
      =\sum_{j=1}^n 
      \langle \partial_j f,\partial_j g\rangle_{L^2(M)\,\bar\otimes\,L^2(M^{op})}
      =\sum_{j=1}^n c_j\,\langle 1\otimes 1,\partial_j g\rangle_{L^2(M)\,\bar\otimes\,L^2(M^{op})}.
  \]

  Since Lipschitz-conjugate variables are bounded by Proposition \ref{prop5} $\xi_j=\partial_j^\ast(1\otimes 1)\in M$, this equals

\[
      \sum_{j=1}^n c_j\,\langle \xi_j,g\rangle_{L^2(M)}.
  \]

  Hence $f=\sum_{j=1}^n c_j\,\xi_j \in M$.  
  This proves the second point.

  \item 
  At this stage we have the two representations

\[
      f=\sum_{j=1}^n c_j X_j^\circ=\sum_{j=1}^n c_j \xi_j.
  \]

  Thus $f$ lies simultaneously in the span of the centered generators and in the span of the conjugate variables.  
  The crucial step is now to renormalize properly.

  \medskip
  We define

\[
      \widehat{f}:=\frac{f}{\sigma},\qquad 
      \delta_f:=\sum_{j=1}^n \frac{c_j}{\sigma}\,\partial_j.
  \]

  Then $\tau(\widehat f)=0$, $\|\widehat f\|_{L^2(M)}=1$, and $\Delta\widehat f=\widehat f$.

  A direct computation using $\partial_j f = c_j.(1\otimes 1)$ and $\sigma^2=\sum_j c_j^2$ shows

\[
      \delta_f(\widehat{f})
      =
      \sum_{j=1}^n \frac{c_j}{\sigma}\,\partial_j\Big(\frac{f}{\sigma}\Big)
      =
      \frac{1}{\sigma^2}\sum_{j=1}^n c_j\,\partial_j f
      =
      \frac{1}{\sigma^2}\sum_{j=1}^n c_j^2.\,(1\otimes 1)
      =
      1\otimes 1.
  \]

  \medskip
  From the second item, we know that

\[
      f=\sum_{j=1}^n c_j\,\xi_j,
      \qquad 
      \xi_j=\partial_j^\ast(1\otimes 1).
  \]

  Set

\[
      h:=\sum_{j=1}^n \frac{c_j}{\sigma}\,\xi_j \;\in\; M.
  \]

  For any $g\in\A$, using the definition of $\delta_f$ we compute:
  \begin{eqnarray*}
      \big\langle 1\otimes 1,\delta_f g\big\rangle_{L^2(M)\bar\otimes\,L^2(M^{op})}
      &=& \left\langle 1\otimes 1,\;
         \sum_{j=1}^n \frac{c_j}{\sigma}\partial_j g
         \right\rangle_{L^2(M\bar\otimes\,L^2(M^{op})} \\
      &=& \sum_{j=1}^n \frac{c_j}{\sigma}\,
         \big\langle 1\otimes 1,\;\partial_j g\big\rangle_{L^2(M)\bar\otimes\,L^2(M^{op})} \\
      &=& \sum_{j=1}^n \frac{c_j}{\sigma}\,
         \big\langle \xi_j,g\big\rangle_{L^2(M)} \\
      &=& \big\langle h,g\big\rangle_{L^2(M)}.
  \end{eqnarray*}

  Thus $1\otimes 1\in\mathrm{Dom}(\delta_f^\ast)$, and by definition of the adjoint,

\[
      \delta_f^\ast(1\otimes 1)=h
      =\sum_{j=1}^n \frac{c_j}{\sigma}\,\xi_j
      =\widehat{f}.
  \]

  \item 
  Finally, we prove the last point. 
  \medskip
  Since $\delta_f^\ast(1\otimes 1)=\widehat{f}$, the adjoint relation
  can be written purely in terms of the trace as follows: for every $g\in\A$,
  \begin{equation}\label{eq:stein-trace}
      \tau(\widehat{f}\,g)
      = 
      \tau\otimes\tau(\delta_f(g)).
  \end{equation}
  This is a Stein identity exactly analogous to the conjugate relation 
  \eqref{conj-rel} for the free difference quotients.

  Applying \eqref{eq:stein-trace} to the polynomials $g=\widehat{f}^m$, $m\ge 0$, we get
  \begin{equation}\label{eq:stein-moment}
      \tau(\widehat{f}^{m+1})
      = 
      \tau\otimes\tau(\delta_f(\widehat{f}^m)),
      \qquad m\ge 0.
  \end{equation}

  On the other hand, since $\delta_f$ is a derivation and $\delta_f(\widehat{f})=1\otimes 1$, the \emph{chain-rule} (or a simple induction) shows that for all $m\ge 1$,

\[
      \delta_f(\widehat{f}^m)
      \ =\ 
      \sum_{k=0}^{m-1} \widehat{f}^k\otimes \widehat{f}^{m-1-k}.
  \]

  Applying $\tau\otimes\tau$ yields

\[
      \tau\otimes\tau\big(\delta_f(\widehat{f}^m)\big)
       = 
      \sum_{k=0}^{m-1} \tau(\widehat{f}^k)\,\tau(\widehat{f}^{m-1-k}),
      \qquad m\ge 1.
  \]

  Combining this with \eqref{eq:stein-moment} we obtain the moment recursion

\[
      \tau(\widehat{f}^{m+1})
      \ =\ 
      \sum_{k=0}^{m-1} \tau(\widehat{f}^k)\,\tau(\widehat{f}^{\,m-1-k}),
      \qquad m\ge 1.
  \]

  By construction, $\tau(\widehat{f})=0$ and $\|\widehat{f}\|_2^2=\tau(\widehat{f}^2)=1$.  
  Thus the moments of $\widehat{f}$ satisfy the standard Catalan recursion which uniquely characterizes the centered semicircular
  distribution of variance $1$. Hence $\widehat{f}$ is semicircular.

\end{enumerate}
\end{proof}
\qed

\begin{flushleft}
We can now state our main theorem which shows that the saturation of the free Poincaré inequality enforces a semicircular direction as well as a free complementation inside the algebra.     
\end{flushleft}

\begin{theorem}[Free Obata rigidity theorem]\label{thm:free-product-lipschitz}
Let $(M,\tau)$ be a tracial $W^\ast$-probability space generated by a self-adjoint
$n$-tuple $X=(X_1,\ldots,X_n)$, and assume that $X$ admits Lipschitz conjugate
variables (Definition~\ref{def6}) and satisfies the curvature-dimension bound
$CD(1,\infty)$.

Suppose there exists a non-zero centered $f\in D(\Delta^{1/2})$, self-adjoint,
saturating the free Poincaré inequality

\[
  \|f\|_{L^2(M)}^2=\mathcal{E}(f).
\]

Then the following hold:

\begin{enumerate}

  \item There exists $U\in O(n)$ such that, writing $X_j^\circ:=X_j-\tau(X_j).1$,

\[
      Y_i:=\sum_{j=1}^n U_{ij}\,X_j^\circ,\qquad i=1,\ldots,n,
  \]

  we have

\[
      f=\|f\|_{L^2(M)}\,Y_1,
  \]

  i.e.\ $Y_1 = f/\|f\|_{L^2(M)}$.

  \item The variable $Y_1$ is a standard semicircular variable (centered of variance $1$).

  \item $Y_1$ is free from the family $(Y_2,\ldots,Y_n)$ under $\tau$.

  \item The von Neumann algebra generated by $X$ satisfies

\[
      W^\ast(X_1,\ldots,X_n)=W^\ast(Y_1,Y_2,\ldots,Y_n).
  \]

  \item Consequently, $(W^\ast(X_1,\ldots,X_n),\tau)$ admits a trace-preserving
  free product decomposition:

\[
      (W^\ast(X_1,\ldots,X_n),\tau)
      \cong
      (W^\ast(Y_1),\tau_{\mathrm{sc}})
      *
      (W^\ast(Y_2,\ldots,Y_n),\tau_N),
  \]

  where $\tau_{\mathrm{sc}}$ is the standard semicircular state on $W^\ast(Y_1)$
  and $\tau_N$ is the restriction of $\tau$ to $N=W^\ast(Y_2,\ldots,Y_n)$.
  In particular,

\[
      (W^\ast(X_1,\ldots,X_n),\tau)
      \cong
      L^\infty([-2,2],\mu_{\mathrm{sc}})
      *
      (W^\ast(Y_2,\ldots,Y_n),\tau_N).
  \]

\end{enumerate}

Moreover, $W^\ast(Y_1)$ is freely complemented in $W^\ast(X_1,\ldots,X_n)$, with
(free) complement $W^\ast(Y_2,\ldots,Y_n)$.
\end{theorem}

\begin{remark}
In the regime $CD(1,\infty)$, it is immediate to see that
$\sigma(\Delta)\subset\{0\}\cup[1,\infty)$.  
When a non-zero saturator of the free Poincaré inequality exists,
Proposition~\ref{prop1} shows that it is an eigenfunction with eigenvalue $1$.
Hence $1$ is the first positive spectral value of $\Delta$, and
$\ker(\Delta-\mathrm{id})$ is indeed the first non-constant eigenspace (see again
the proof of Theorem~\ref{thm:BLP-Lipschitz}, which shows that $\ker(\Delta)$
consists only of scalar multiples of the identity).
\end{remark}

\begin{corollary}\label{cor:multi-free-obata}
Assume the hypotheses of Theorem~\ref{thm:free-product-lipschitz}, and in addition suppose that
the first eigenspace of the free Laplacian

\[
    E_1 := \ker(\Delta - \mathrm{id}) \subset L^2_0(M)
\]

is finite-dimensional of dimension $r \ge 1$.

By Corollary~\ref{cor1}, $E_1$ is spanned by a centered self-adjoint saturator of the free Poincaré inequality (reverse inclusion: \textit{eigenfunction}$\implies$ \textit{saturator} being immediate by duality). Since $E_1$ is assumed finite-dimensional, we may choose $f_1,\ldots,f_r$ to be an orthonormal basis of $E_1$. Then the following hold:

\begin{enumerate}

  \item There exists $U \in O(n)$ such that, writing $X_j^\circ := X_j - \tau(X_j).1$,

\[
      Y_i := \sum_{j=1}^n U_{ij}\,X_j^\circ,\qquad i=1,\ldots,n,
  \]

  we have

\[
      f_k = Y_k,\qquad k=1,\ldots,r.
  \]

  \item The variables $Y_1,\ldots,Y_r$ are standard semicircular variables (centered of variance $1$).

  \item The family $(Y_1,\ldots,Y_r)$ is free from the family $(Y_{r+1},\ldots,Y_n)$ under $\tau$.

  \item The von Neumann algebra generated by $X$ satisfies

\[
      W^\ast(X_1,\ldots,X_n)
      =
      W^\ast(Y_1,\ldots,Y_r,Y_{r+1},\ldots,Y_n).
  \]

  \item Consequently, $(W^\ast(X_1,\ldots,X_n),\tau)$ admits a trace-preserving
  free product decomposition:

\[
      (W^\ast(X_1,\ldots,X_n),\tau)
      \cong
      (W^\ast(Y_1,\ldots,Y_r),\tau_{\mathrm{sc}}^{*r})
      *
      (W^\ast(Y_{r+1},\ldots,Y_n),\tau_N),
  \]

  where $\tau_{\mathrm{sc}}^{*r}$ is the free product of $r$ copies of the standard
  semicircular state, and $\tau_N$ is the restriction of $\tau$ to
  $N=W^\ast(Y_{r+1},\ldots,Y_n)$. In particular,

\[
      (W^\ast(X_1,\ldots,X_n),\tau)
      \cong
      L(\mathbb{F}_r)
      *
      (W^\ast(Y_{r+1},\ldots,Y_n),\tau_Y).
  \]

\end{enumerate}

Moreover, $W^\ast(Y_1,\ldots,Y_r)$ is freely complemented in
$W^\ast(X_1,\ldots,X_n)$, with (free) complement $W^\ast(Y_{r+1},\ldots,Y_n)$.
\end{corollary}

\begin{remark}
The hypothesis that we have a first eigenspace $\dim E_1 = r < \infty$ is an additional very strong \emph{spectral assumption}.
It holds automatically in the semicircular/$q$--Gaussian algebra case~ \cite{Bozejko}, since $\Delta$
is exacty the \emph{number operator} on the Fock/q-Fock space and has a discrete spectrum with finite multiplicities.
For general free Gibbs states, and a fortiori for the more general case we considered here of \emph{Lipschitz conjugate variables}, compactness of the resolvent (and hence
finite-dimensionality of eigenspaces) is currently unknown an open problem; see discussion in 
Section~\ref{Open}.

Moreover, by the affine rigidity result (Corollary~\ref{affine}), every
first--eigenfunction is of the form
$f = \sum_{j=1}^n u_j X_j^\circ$ with $u \in \mathbb{R}^n$.
Thus the coefficient vectors $u_1,\ldots,u_r$ associated to an orthonormal basis
of $E_1$ are orthonormal in $\mathbb{R}^n$, and therefore

\[
r = \dim E_1 \le n.
\]

Consequently, Corollary~\ref{cor:multi-free-obata} has to be understood as a conditional rigidity
statement: whenever the first eigenspace is finite-dimensional, its extremal
directions necessarily form a free semicircular family and yield an
$L(\mathbb{F}_r)$ free product factor.
\end{remark}

\begin{proof}[Proof of Theorem~\ref{thm:free-product-lipschitz}]
We divide the proof into four steps.

\medskip
\noindent\textit{Step 1.}
By Corollary~\ref{cor1}, we may assume that $f$ is self-adjoint. Since $f$ is a non-zero
centered self-adjoint saturator of the free Poincaré inequality under $CD(1,\infty)$,
Lemma~\ref{lem:affine-conjugates-semicircle}(i)–(ii) yields coefficients
$c_1,\ldots,c_n\in \mathbb R$ such that

\[
  f=\sum_{j=1}^n c_j\,X_j^\circ
  \qquad\text{and}\qquad
  f=\sum_{j=1}^n c_j\,\xi_j,
\]

where $X_j^\circ:=X_j-\tau(X_j).1$ and $(\xi_1,\ldots,\xi_n)$ are the (bounded, i.e. in $M$ and not merely in $L^2(M)$ by~\ref{prop5}) Lipschitz
conjugate variables of $X$. Moreover, by Lemma~\ref{lem:affine-conjugates-semicircle}(i),

\[
  \|f\|_{L^2(M)}^2=\sum_{j=1}^n c_j^2>0.
\]

Set $c=(c_1,\ldots,c_n)$, for which we have  $\|c\|_2=(\sum_{j=1}^nc_j^2)^{1/2}=\|f\|_{L^2(M)}>0$, and $u=c/\|c\|_2$.
Choose $U\in O(n)$ with first row $U_{1\ast}=u$ and define

\[
  Y_i:=\sum_{j=1}^n U_{ij}\,X_j^\circ,
  \qquad i=1,\ldots,n.
\]

Then

\[
  f=\sum_{j=1}^n c_j X_j^\circ
    =\|c\|_2\sum_{j=1}^n u_j X_j^\circ
    =\|f\|_{L^2(M)}\,Y_1,
\]

so that

\[
  Y_1=\frac{f}{\|f\|_{L^2(M)}}.
\]

Since $U$ is orthogonal, the linear change of variables $X^\circ\mapsto Y$ is invertible, and hence

\[
  W^\ast(X_1,\ldots,X_n)
  =W^\ast(X_1^\circ,\ldots,X_n^\circ)
  =W^\ast(Y_1,\ldots,Y_n).
\]

\medskip
\noindent\textit{Step 2.} Since conjugate variables are invariant under adding scalar multiples of the identity to the generators, the conjugate system associated with $X^\circ=(X_1^\circ,\ldots,X_n^\circ)$ is still $(\xi_1,\ldots,\xi_n)$. 

\medskip

By Voiculescu’s results on derivations and conjugate variables under real invertible linear changes of
variables (\cite[Remark~6.6, Lemma~6.7, Corollary~6.8]{V}), since $Y = UX^\circ$, the
conjugate system transforms by the transpose of the inverse matrix, i.e. the conjugate variables of $Y$ are $\xi_{Y_i} = \sum_{j=1}^n
(U^{-1})_{ji}\,\xi_j$. Since $U$ is orthogonal, $U^{-1}=U^{T}$, i.e.
$(U^{-1})_{ji}=U_{ij}$, so that $
    \xi_{Y_i}=\sum_{j=1}^n
U_{ij}\,\xi_j$.
In particular, 
\begin{equation}
    \xi_{Y_1}=\sum_{j=1}^n
U_{1j}\,\xi_j=\sum_{j=1}^n u_j\,\xi_j
\end{equation}
\bigbreak
Using $f=\sum_{j=1}^n c_j\xi_j$ and $c=\|c\|_2\,u$, we obtain

\[
  f=\sum_{j=1}^n c_j\,\xi_j
   =\|c\|_2\sum_{j=1}^n u_j\,\xi_j
   =\|f\|_{L^2(M)}\,\xi_{Y_1}.
\]

Comparing with $f=\|f\|_{L^2(M)}\,Y_1$, we conclude that

\[
  \xi_{Y_1}=Y_1.
\]

\medskip

\medskip
\noindent\textit{Step 3.} From Step 2, we have the conjugate of $Y_1$ which is equal to itself, i.e. 
\[
\xi_{Y_1} = Y_1.
\]

By Voiculescu’s free Cramér--Rao equality characterization~\cite[Proposition~6.9]{V}, which is also a Stein-type characterization, this immediately implies that $Y_1$ is a centered semicircular variable with variance $1$ and free from $Y_2,\ldots,Y_n$, i.e. $\mathbb{C}[Y_1]$ free from the unital $*$-algebra $N_0 := \mathbb{C}\langle Y_2, \dots, Y_n \rangle$. Then, by Kaplansky’s density theorem (cf. Mingo--Speicher\cite[Proposition~5]{MS}), $N_0$ is strongly dense in the von Neumann algebra $N = W^\ast(Y_2, \dots, Y_n)$ and freeness extend so that
$Y_1$, i.e. $\mathbb{C}[Y_1]$, and thus $W^*(Y_1)$, is free from $N = W^\ast(Y_2, \dots, Y_n)$, proving items~\textnormal{(2)} and~\textnormal{(3)}.

\medskip
\noindent\textit{Step 4.}
Since $U$ is orthogonal,

\[
  W^\ast(X_1,\ldots,X_n)=W^\ast(Y_1,Y_2,\ldots,Y_n),
\]

and $Y_1$ is standard semicircular, free from $W^\ast(Y_2,\ldots,Y_n)$, we therefore have by the Universal property characterizing the free product (see.~\cite{VD} for details):

\[
  (W^\ast(X_1,\ldots,X_n),\tau)
  \cong
  (W^\ast(Y_1),\tau_{\mathrm{sc}})\ *\ (W^\ast(Y_2,\ldots,Y_n),\tau_N),
\]

where $\tau_{\mathrm{sc}}$ is the standard semicircular state on $W^\ast(Y_1)$ and $\tau_N$ is the
restriction of $\tau$ to $N=W^\ast(Y_2,\ldots,Y_n)$. Since
$(W^\ast(Y_1),\tau_{\mathrm{sc}})\cong L^\infty([-2,2],\mu_{\mathrm{sc}})$, item~(5) follows, and the
theorem is proved.
\end{proof}
\qed

\begin{proof}[Proof of Corollary~\ref{cor:multi-free-obata}]
Let $f_1,\ldots,f_r$ be a self-adjoint (and bounded, i.e. in $M$) orthonormal basis of $E_1$.  
By Proposition~\ref{prop1}, each $f_k$ is a centered self-adjoint saturator of the
free Poincaré inequality. Applying Theorem~\ref{thm:free-product-lipschitz} to each
$f_k$ yields the conclusion.
\end{proof}
\qed
\begin{flushleft}
As another direct consequence, the free product decomposition obtained above
yields several structural properties of the resulting von Neumann algebras,
in both the one-dimensional and multi-dimensional splitting regimes.  In
particular, this recovers the corresponding results (in our particular case) of
Dabrowski~\cite{Dab10} by more direct arguments that avoid free
Fisher/entropy techniques.
\end{flushleft}
\begin{flushleft}
We first record a general factoriality criterion for free products, which is well known and can for example be found in Ioana's lecture notes~\cite[Exercise~5.24]{Ioana}.    
\end{flushleft}

\begin{lemma}\label{lem:free-product-factor}
Let $(M_1,\tau_1)$ and $(M_2,\tau_2)$ be tracial von Neumann algebras such that
$M_1$ is diffuse and $M_2\neq\C 1$. Let $(M,\tau)=(M_1,\tau_1)*(M_2,\tau_2)$ be
their tracial free product. Then $M$ is a $\mathrm{II}_1$ factor.
\end{lemma}

\begin{corollary}\label{cor:nonamenable-nonGamma}
Assume the hypotheses of Theorem~\ref{thm:free-product-lipschitz} or of
Corollary~\ref{cor:multi-free-obata}, and suppose $n\ge 2$. Then, for
$M=W^\ast(X_1,\ldots,X_n)$, we have:
\begin{enumerate}
  \item $M$ is a $\mathrm{II}_1$ factor;
  \item $M$ is non-amenable.
\end{enumerate}
\end{corollary}

\begin{proof}
Under the hypotheses of Theorem~\ref{thm:free-product-lipschitz} or
Corollary~\ref{cor:multi-free-obata}, we obtain a free product
decomposition

\[
(M,\tau)
\simeq
(W^\ast(Y_1),\tau_{\mathrm{sc}})
*
(N,\tau_N),
\]

where $Y_1$ is a centered semicircular variable of variance~$1$, and
$N=W^\ast(Y_2,\ldots,Y_n)$ with restricted trace (or $W^\ast(Y_{r+1},\ldots,Y_n)$ in the
multi-dimensional case).

By Voiculescu’s result~\cite[Corollary~4.7]{V}, finite free Fisher
information (implying finite free entropy by free log-Sobolev, cf.~\cite[Proposition~7.9]{V}) ensures that
each $X_j$ has a diffuse distribution. Since translation preserves
diffuseness, each centered variable $X_j^\circ$ is diffuse. The
orthogonal change of variables

\[
(X_1^\circ,\ldots,X_n^\circ)\mapsto (Y_1,Y_2,\ldots,Y_n)
\]

shows that each $(Y_2,\ldots,Y_n)$ has conjugate variable, and so each component is diffuse. For $n\ge 2$, at least one $Y_i$ is non-trivial, so
$N=W^\ast(Y_2,\ldots,Y_n)$ is non-trivial and contains a diffuse element.

\smallskip
\begin{enumerate}
\item Since $W^\ast(Y_1)$ is diffuse and $N\neq\C 1$, Lemma~\ref{lem:free-product-factor}
applies, and we conclude that $M$ is a $\mathrm{II}_1$ factor.

\item Identifying $W^*(Y_1)\cong L^\infty([-2,2],\mu_{\mathrm{sc}})$, we then easily produce a Haar unitary in $W^*(Y_1)$.
\newline
Since $N$ contains a diffuse element, it contains a diffuse abelian subalgebra 
$A\subset N$, and we can find $v$ a Haar unitary in $A$ (see Mingo-Speicher~\cite[Theorem~6]{SM} for details).  In the free product decomposition
$M\simeq W^*(Y_1)*N$, the Haar unitaries $u$ and $v$ are free, so
$W^*(u,v)\cong L(\mathbb F_2)\subset M
$. Since $M$ is a $\mathrm{II}_1$ factor with separable predual, if it were amenable then by Connes' classification it would be isomorphic to the hyperfinite $\mathrm{II}_1$ factor $R$, but every $\mathrm{II}_1$ subfactor of $R$ is again isomorphic to $R$, so hyperfinite/amenable, whereas $L(\mathbb F_2)$ is not. Thus $M$ is non-amenable.

\end{enumerate}
\end{proof}
\qed
\begin{flushleft}
A more striking point is the following. Popa’s breakthrough result shows that in a $\mathrm{II}_1$ factor, diffuse amenable subalgebras which are freely complemented are automatically
maximal amenable.
\end{flushleft}
\begin{theorem}[Popa~\cite{Popa}]
If $M$ is a $\mathrm{II}_1$ factor, then any diffuse amenable subalgebra
$B \subset M$ that is freely complemented in $M$ is maximal amenable in $M$.
\end{theorem}

\begin{corollary}
Assuming the hypotheses and notations of
Theorem~\ref{thm:free-product-lipschitz} or of
Corollary~\ref{cor:multi-free-obata}, the subalgebras $W^\ast(Y_1)$ or $W^\ast (Y_i),\:1\leq i\leq r$ are
maximal amenable in $W^\ast(X_1,\ldots,X_n)$, and hence in particular MASAs.
\end{corollary}

\begin{remark}
Under the same hypotheses, one also has much stronger rigidity properties
for the free product decomposition of Theorem~\ref{thm:free-product-lipschitz}.

\begin{enumerate}
  \item \emph{Non-$\Gamma$.}
  In the decomposition of $M$ given by Theorem~\ref{thm:free-product-lipschitz},
  both free components are diffuse and non-trivial.  Hence Houdayer’s
  $\Gamma$-stability theorem for free products \cite[Theorem~A]{Houdayer-Gamma}
  applies to the inclusion $W^\ast(Y_1)\subset M$, and it follows that $M$
  does not have property~$\Gamma$.

  \item \emph{Absence of Cartan subalgebras.}
  By the same decomposition, $M$ is a free product of non-trivial, infinite-dimensional
  tracial von Neumann algebras.  Therefore Ioana’s result on the absence
  of Cartan subalgebras in free products \cite[Corollary~1.5]{Ioana-Cartan-AFP}
  applies, and $M$ admits no Cartan subalgebra.
\end{enumerate}
\end{remark}

\section{Open questions}\label{Open}
\subsection{Compactness of the resolvent and discrete spectrum}

In the classical setting, for a \emph{regular}
log-concave probability measure
$\mu\propto e^{-V}dx$ on $\mathbb{R}^n$ with generator

\begin{equation}\label{langevin}
   \mathcal L_{V}u=\Delta u-\nabla V\cdot\nabla u, 
\end{equation}

the operator $-\mathcal L_{V}$ is essentially self-adjoint and generates a symmetric
Markov semigroup $(P_t)_{t\ge0}$ on $L^2(\mu)$ with Dirichlet form
$\mathcal{E}(f)=\int |\nabla f|^2\,d\mu$.  A fundamental fact is the
equivalence
\begin{equation}\label{eq:comp-resolvent-discrete}
(\alpha+\mathcal L_{V})^{-1}\ \text{compact}
\quad\Longleftrightarrow\quad
\text{$-\mathcal L_{V}$ has discrete spectrum},
\end{equation}
see \cite{Davies,ReedSimonIV}.  Compactness may follow from several
functional-analytic criteria, such as Rellich--Kondrachov compact
embeddings, hypercontractivity, or Nash inequalities
\cite{Ledoux,Davies}.  Under these conditions, $-\mathcal L_{V}$ admits a complete
orthonormal basis of eigenfunctions with Rayleigh variational
characterization, and for confining potentials
(e.g.\ uniformly convex $\psi$) eigenfunctions decay exponentially
\cite{Simon,ReedSimonIV}.

\medskip

In the free setting, let $M=W^\ast(X_1,\ldots,X_n)$ be generated by
self-adjoint variables , with (possibly Lipschitz) conjugate variables, and consider the conservative generator

\[
\Delta=\sum_{j=1}^n \partial_j^{\,*}\,\bar\partial_j.
\]

By Dabrowski \cite{Dab10,Dab14}, $\Delta$ generates a completely Dirichlet
form on $L^2(M,\tau)$, and its resolvents
$\eta_\alpha=\alpha(\alpha+\Delta)^{-1}$ are completely positive
contractions with $\mathrm{Ran}(\eta_\alpha)\subset\Dom(\bar\partial)$.

The spectral theory of $\Delta$ is largely open.  In the semicircular
case, the free Ornstein--Uhlenbeck generator is the number operator on
full Fock space and has pure point spectrum
\begin{equation}\label{eq:free-OU-spectrum}
\mathrm{Sp}(-\mathcal{L}_{V_0})=\{0,1,2,\ldots\},
\end{equation}
with eigenprojections onto the Wigner chaoses \cite{BS,Di2}.  This
algebraic (here we have polynomial eigenfunctions).
diagonalization implies compact resolvent.

At present, compactness of the resolvent is known only in the
semicircular and $q$-Gaussian cases ($|q|<1$), the latter via
ultracontractivity \cite{Bozejko} (e.g.\ free Gibbs states with uniformly
convex potentials), discreteness of the spectrum remains unknown.  A
non-commutative analogue of the classical criteria would suffice:
\begin{enumerate}
  \item \emph{Compact resolvent:} a free Rellich-type compact embedding
        $\Dom(\Delta^{1/2})=\Dom(\bar\partial)\hookrightarrow L^2(M,\tau)$
        would imply discrete spectrum.
  \item \emph{Ultracontractivity:} if $e^{-t\Delta}$ maps $L^2$ into
        $L^\infty(M):=M$ for some $t>0$, then $e^{-t\Delta}$ (and hence
        $(\alpha+\Delta)^{-1}$) is compact
        \cite{CiprianiSauvageot2003}.
\end{enumerate}

Extending discreteness to general
free Gibbs states remains an open problem.
We also note a related question, raised already at the AIM 2006 workshop on free analysis
\cite{Aim}, 
asking whether the semigroup $\exp(-t\,\Delta)$ is compact in the limit
$n\to\infty$.  Since compactness of $e^{-t\Delta}$ is equivalent to
discrete spectrum of $\Delta$, this question is closely related to the
free spectral problem discussed above.  Even in the simplest perturbative
case $P=X_i^2+t_i q_i(X_1,\ldots,X_m)$ nothing is known. Moreover, even in regimes where free monotone transport shows
that the limiting von Neumann algebra is isomorphic to a free group factor
(see, e.g., \cite{GS}), there is currently no known operator (unitary or otherwise) that intertwines or compares the associated Dirichlet forms in a way strong enough to transfer compactness or discreteness of the spectrum.

\subsection{Free Ricci curvature and almost co-associativity}\label{freericci}
In Riemannian geometry, curvature manifests itself through commutation identities.
At the geometric level, this is expressed by the fundamental relation for smooth vector fields $X,Y,Z\in \mathfrak{X}(M)\times \mathfrak{X}(M)\times \mathfrak{X}(M)$

\[
R(X,Y)Z \;=\; \nabla_X \nabla_Y Z \;-\; \nabla_Y \nabla_X Z \;-\; \nabla_{[X,Y]} Z ,
\]
where $\nabla$ is the Levi-Civita connection

In local coordinates $(x^1,\dots,x^n)$ on a open set $U\subset M$, let
\[
\partial_i := \frac{\partial}{\partial x^i}, \qquad \nabla_i := \nabla_{\partial_i}.
\]
Since coordinate vector fields commute, $[\partial_i,\partial_j] = 0$, the Riemann curvature tensor acts on a vector field $X = X^k \partial_k$ as
\[
[\nabla_i, \nabla_j] X^k = R^k{}_{\ell ij} X^\ell,
\]
where $R^k{}_{\ell ij}$ are the coordinate components of the curvature tensor.
This expresses the failure of covariant derivatives to commute in components.
 In the classical Bakry--Émery framework on $\mathbb{R}^n$, the
underlying Riemannian metric is flat (so the geometric Ricci tensor vanishes),
and all curvature information comes from the potential $V$ through its Hessian
$\nabla^2 V$.  This is encoded in the commutation identity

\[
[\nabla,\mathcal{L}_V] \;=\; -\,\nabla^2 V,
\]

where $\mathcal{L}_V = \Delta - \nabla V\cdot\nabla$ is the Langevin generator
with invariant measure $\mu\propto e^{-V}dx$.  Thus, in this setting, the
``curvature'' in the sense of Bakry--Émery is purely of potential type: it is
entirely captured by $\nabla^2V$, with no additional geometric Ricci
contribution.

\smallskip

It is therefore natural to ask what should play the role of \emph{free Ricci
curvature}.  Dabrowski's work~\cite{Dab14v2} suggests this direction, though
perhaps not intended at first to carry geometric meaning, since it was developed
primarily to obtain the domain and commutation control needed to get mild
solutions of free SPDEs.  In particular, his framework of \emph{almost
co-associative} derivations \cite[Definition~20]{Dab14v2} provides a natural
candidate, which may be viewed as a kind of \emph{coupled} derivations taking
values in the coarse correspondence, and which is significantly more general
than the free difference quotient situation considered in the present work.

\bigbreak

Before describing the setting more precisely, let us first introduce some
notation: for $U\in L^2(M\bar \otimes M^{op})$ and
$C\in M\bar\otimes M\bar\otimes M^{op}$, one defines
$U\# C\in L^2(M\bar\otimes M\bar\otimes M^{op})$ by extending the rule

\[
(a\otimes b)\#(d\otimes e\otimes f)
  := ad\otimes e\otimes bf,
\]

which reflects the natural bimodule structure.  Equivalently, since $U\#1$ is
the canonical embedding into the threefold tensor product, one may write
$U\# C = (U\#1)\,C$.

\smallskip

In the sense of Dabrowski~\cite{Dab14v2}, almost coassociativity requires a
precise compatibility between the domains of the derivations and their tensor
extensions.  Let

\[
\delta=(\delta_1,\ldots,\delta_n),\qquad
\tilde\delta=(\tilde\delta_{n+1},\ldots,\tilde\delta_M)
\]

be two families of real closable derivations
$\delta: \Dom(\delta)\to (L^2(M)\bar\otimes L^2(M^{op}))^{\oplus_n}$ and
$\tilde{\delta}: \Dom(\tilde{\delta})\to (L^2(M)\bar\otimes L^2(M^{op}))^{\oplus_n}$
satisfying the standing domain assumptions of~\cite[Definition~20]{Dab14v2}
(both $\Dom({\delta})$ and $\Dom(\tilde{\delta})$ weakly dense $*$-subalgebras
of $M$).  We say that $\delta$ is almost coassociative with respect to
$\tilde\delta$, with defect tensor

\[
C=(C^k_{i,j},C^k_{j,i})\in (M_0\bar\otimes M_0\bar\otimes M^{op}_0)^{2nM^2},
\]

if the following holds.  For every $x\in \Dom(\delta)\cap \Dom(\tilde\delta)$
such that $\delta_j(x)\in \Dom(\overline{\tilde\delta\otimes 1})\cap
\Dom(\overline{1\otimes\tilde\delta})$ for all $j=1,\ldots,n$, and similarly
$\tilde\delta_i(x)\in \Dom(\overline{\delta\otimes 1})\cap
\Dom(\overline{1\otimes\delta})$ for all $i=1,\ldots,M$, the co-associator
identities

\[
(\delta_j\otimes 1)\circ\tilde\delta_i(x)
-
(1\otimes\tilde\delta_i)\circ\delta_j(x)
=
\sum_{k=N+1}^M \tilde\delta_k(x)\# C^k_{j,i},
\]

and

\[
(\tilde\delta_i\otimes 1)\circ\delta_j(x)
-
(1\otimes\delta_j)\circ\tilde\delta_i(x)
=
\sum_{k=1}^N \delta_k(x)\# C^k_{i,j},
\]

hold.

\medskip

If $\delta=\tilde\delta$, we simply say that $\delta$ is
\emph{almost co-associative}.  Thus the \emph{co-associator} of $\delta$ and
$\tilde\delta$ is expressed as a controlled linear combination of the
gradients $\tilde\delta_k(x)$, with coefficients in
$M\bar\otimes M\bar\otimes M^{op}$.  This is directly analogous to the
classical fact that the commutator of vector fields is a linear combination of
vector fields with scalar coefficients; in this sense, almost coassociativity
provides a non-commutative \emph{Lie algebra} structure on derivations.  The
defect tensor $C$ plays the role of an operator-valued curvature for this
co-derivation geometry.

\smallskip

One of the main technical achievements of~\cite{Dab14v2} is that, under
additional Sobolev-type regularity and suitable ($\Gamma_1$--type) bounds,
almost co-associativity allows one to derive an \emph{almost commutation}
relation between the gradient and the associated second-order operator.  More
precisely, if $\Delta=\delta^\ast\bar\delta$ is a generator in divergence form
built from a family of derivations $\delta$, and if $\tilde\delta$ is another
family almost coassociative with respect to $\delta$, then there exists a
bounded operator

\[
H : (L^2(M)\bar\otimes L^2(M^{op}))^{\oplus_n} \longrightarrow
    (L^2(M)\bar\otimes L^2(M^{op}))^{\oplus_n}
\]

such that

\[
\Delta^{\otimes}\circ\tilde\delta(x)
-
\tilde\delta\circ\Delta(x)
=
H(\tilde\delta(x)),
\qquad x\in \Dom(\tilde\delta\circ\Delta),
\]

where

\[
\Delta^{\otimes}
  = \overline{\Delta\otimes\mathrm{id}+\mathrm{id}\otimes\Delta}
\]

is the natural (tensor extension) Laplacian on
$(L^2(M)\bar\otimes L^2(M^{op}))^{\oplus_n}$.  This identity is a free analogue
of the classical commutation relation evoked before, with the bounded operator
$H$ playing the role of a curvature operator acting on gradients.  The
coefficients of $H$ are extremely complicated and expressed explicitly in terms
of the defect tensor $C$, the adjoints $\delta_i^\ast(1\otimes 1)$, and the
bounded maps $(1\otimes\tau)\circ\delta_i$
(see~\cite[Section~2.1.5]{Dab14v2} and in particular~\cite[Corollary~32]{Dab14v2}).

\smallskip

The present work focuses on the case where the underlying non-commutative space
is the free analogue of Euclidean space $(\mathbb{R}^n,\lvert\,\cdot\,\rvert)$,
so there is no geometric Ricci curvature built into the background von Neumann
algebra $M=W^\ast(X_1,\ldots,X_n)$.  In this flat situation, the free
difference quotients are exactly co-associative, and therefore the defect
tensor $C$ vanishes.  Consequently, all curvature comes from the potential,
through the conjugate variables $\xi$ and their Jacobian $\mathscr{J}\xi$, just
as in the classical Bakry--Émery theory on $\mathbb{R}^n$ where $\mathrm{Ric}=0$
and the only curvature is $\nabla^2V$.

\smallskip

By contrast, in the full almost coassociative framework of Dabrowski, a
non-trivial defect tensor $C$ provides a natural \emph{free Ricci
curvature}, independent of the potential.  In that more general setting, the
commutation relation

\[
\Delta^{\otimes}\tilde\delta - \tilde\delta\Delta = H(\tilde\delta)
\]

contains contributions both from the potential (via $\mathscr{J}\xi$) and from
the geometric defect $C$, in direct analogy with the decomposition
$\mathrm{Ric}_V = \mathrm{Ric} + \nabla^2V$ in the classical Bakry--Émery
identity.  Extending the analytic part of the present work (curvature--dimension
estimates, Poincaré inequalities, and rigidity phenomena) to this more general
setting would therefore lead to a fully non-commutative curvature--dimension
theory in which both geometric curvature (encoded by $C$) and potential
curvature (encoded by $\mathscr{J}\xi$) contribute equally to the curvature
term.

\section{Acknowledgments}
This work was supported by ISF grant 2574/24 and NSF-BSF grant 2022707.
C.\,P.~Diez also acknowledges support from the Ministry of Research, Innovation and Digitalization (Romania), grant CF-194-PNRR-III-C9-2023. The author would like to express his gratitude to Dr.~David Jekel for many helpful comments and discussions, which in particular led to a simplification of the main theorem. He is also indebted to Professor Liran Rotem for encouragement and insightful remarks during the preparation of this work.

\end{document}